\documentclass[AMA,STIX1COL]{WileyNJD-v2}

\articletype{Article Type}%

\received{-}
\revised{-}
\accepted{-}

\usepackage{graphicx}
\usepackage{xcolor}
\usepackage{dsfont}

\usepackage{url}

\def\qed{\hbox{${\vcenter{\vbox{		 
   \hrule height 0.4pt\hbox{\vrule width 0.5pt height 6pt
   \kern5pt\vrule width 0.5pt}\hrule height 0.4pt}}}$}}

\def\bE{\mathbb E}

\def\bP{\mathbb P}

\def\eps{\varepsilon}

\def\ov{\overline}
\def\uno{\mathds{1}}

\newcommand{\ste}[1]{\textcolor{black}{#1}}
\newcommand{\sO}[1]{\textcolor{black}{#1}}

\begin{document}

\title{Some aspects of the Markovian SIRS epidemic on networks and its mean-field approximation}

\author[1]{Stefania Ottaviano*}
\author[1]{Stefano Bonaccorsi}

\authormark{S. Ottaviano \textsc{et al}}

\address[1]{\orgdiv{Mathematics Department}, \orgname{University of Trento}, \orgaddress{\state{Via Sommarive 14, 38123 Povo}, \country{Trento, Italy}}}



\corres{*Corresponding author, \email{stefania.ottaviano@unitn.it}}


\abstract[Summary]{We study the spread of an SIRS-type epidemic with vaccination on network. Starting from an exact Markov description of the model, we investigate the mean epidemic
lifetime by providing a sufficient condition for fast extinction {that depends on the model parameters and the topology of the network.}
\ste{Then, we pass to consider a first-order mean-field approximation of the exact model and its stability properties, by relying on the graph-theoretical notion of equitable partition}. \ste{ In the case of graphs possessing this kind of partition, we prove that the endemic equilibrium can be computed by using a lower-dimensional dynamical
system. 
Finally, in the special case of regular graphs, we investigate the domain of attraction of the endemic equilibrium}. }


\jnlcitation{\cname{%
\author{S. Ottaviano}, and 
\author{S. Bonaccorsi}}
 (\cyear{2020}), 
\ctitle{Some aspects of the Markovian SIRS epidemic on networks and its mean-field approximation}, \cjournal{Math. Meth. Appl. Sci}, \cvol{-}.}

\maketitle


\section{Introduction}\label{sec1}

The spread and persistence of infectious diseases are a result of the complex
interactions between individual units (e.g. people, city, county, etc),
disease characteristics and possible control policies. 
Consequently, the aim of many mathematical models is to gain insight into how diseases
transmit and to identify the most effective strategies for their prevention and
control. 
%
%
%
Vaccination is considered to be the most effective intervention policy as well as a cost-effective strategy to reduce
both the morbidity and mortality of individuals.
Over the past few decades a high variety of compartmental models, where the population is divided
into different classes (compartments), depending on the stage of the disease, have been formulated. 
A relevant amount of these models, including those that incorporate a vaccination strategy, assumes a \emph{homogeneous mixing approximation} \cite{kribs2000simple, alex2004vaccination,elbasha2006theoretical, sun2010global, cai2018global}. Basically, individuals in the population interact with each other completely at random (with no preferential interaction). 
Although the simplicity of the model allows to include more specific characteristics, such as birth and deaths, vaccination by age  etc., 
the homogeneous mixing assumption ignores details such as geographical location, presence of community structures, or the specific role of each individual in the contagion spreading. 
However,
the underlying contact structure of the population plays a crucial role in the spreading of the epidemics     
\cite{boccaletti2006, pastor2015epidemic,kiss2017mathe}.

 Epidemic
models have also been used to describe a wide range of others phenomena.
like social behaviors, diffusion of information, computer viruses etc., indeed the dynamical behavior of these phenomena can be described by the same type of equations, although their basic mechanisms may differ \cite{pastor2015epidemic}.
For example, 
networks through which agents
communicate with one another are frequently used to propagate electronic viruses. Thus, epidemiological modeling method can help to understand how such
viruses spread on a network for building proper effective strategies to stem the viral prevalence, e.g. to implement antivirus techniques \cite{yang2017heterogeneous,balthrop2004technological}. 
For a review on epidemics models on networks see, e.g., \cite{PietSurvey,danon2011networks}. \ste{Spatially extended epidemiological processes are described also via superdiffusion, e.g., in \cite{skwara2018superdiffusion}.}

In our model, we classify each individual in the population according to her state:  susceptible, infected or recovered. 
An individual in the susceptible state can be infected if she is in contact with any infected individual (equivalently, they are adjacent nodes in the network). 
After the infection is over, the individual enters in the recovery state, and while in the recovery state, she cannot undergo to a new infection. 
However, in this  work we analyze a 
model where the recovery state is not permanent, hence the individual returns, after an exponentially distributed time,
to the susceptible state.

Moreover, a further mechanism exists that change the state of an individual, that is vaccination. Vaccination takes place for susceptible individuals who are moved directly to the recovery state.
We do not add a compartment for the vaccinated individuals, not distinguishing the vaccine-induced immunity from the natural one acquired after the virus contraction.
In several examples in applications, actually, vaccination does not confer a long-life immunity (in the field of infectious disease, think, e.g., to influenza, diphtheria, pertussis and pneumococcal vaccine).


Overall, the model we consider can be classified as a SIRS \emph{susceptible-infected-removed-susceptible} model \emph{with vaccination}, on networks, that we shall refer to with SIRS$_{v}$. 
Moreover, we adopt an individual (node)-based approach, %
see also \cite{Turri, yang2017heterogeneous} for SIRS-type node-based models. 
 As opposite, a large part of the literature 
consider models in which the structure of the network is simplified by using
 a degree-based mean-field (DBMF) approach, \ste{ \cite{chen2014SirsVaccNet,liu2019SIRSnetl,liu2017analysis,yu2015dynamical}, where all nodes with the same degree are assumed to be statistically equivalent.} 
\ste{Thus, these kinds of models only reflect the evolution in time of the
 fraction of nodes with a certain degree $k$ in each compartment, while neglecting the
states of each single individual. 
 This leads to a loss of detailed features of network topologies resulting in difficulties for a deep understanding of the effect of a particular topology on the infection propagation.}

\subsection{Outline and main results.}

 In Sec.\ref{exact}, we start considering the exact stochastic SIRS model with vaccination. As stated before, we have a population of $N$ individuals where each of them can be classified in one of the three states, $S$, $I$ or $R$. Therefore, the process describing the spreading of the epidemics among the population counts $3^N$ possible states.
Our system evolves as a continuous-time Markov chain: all the involved processes, vaccination, infection, recovery and loss of immunity, are thought as independent Poisson processes
each with its own rate (that allows to jump from a state to another). This approach describes the global change in the state probabilities of the network exactly. 


In this context, we investigate the mean time in which the epidemic is active (at least one node is infected),  trying to understand in which way the network topology, and the parameters of the model, are responsible for a quick epidemic extinction. 
We also provide some numerical investigations to assess the role of the {immunity-loss} parameter in the extinction mean time.

The exponential growth of the state space with $N$ makes the search for solution neither analytically
nor computationally tractable, except for very small networks. Hence, it is necessary to derive an approximation of the original model. A direct approach for deriving an approximate model is to start from a node-level description of the underlying exact stochastic process (Sec.\ref{nodelev}), as proposed in
\cite{PVM_GEMF}, and then, through a first-order mean-field approximation, 
obtain a set of $3N$ nonlinear
differential equations, specifying the state probabilities of each node (Sec. \ref{Secmeanf}). Basically, we consider an extension of the N-intertwined mean-field approximation (NIMFA), provided for the SIS and SIR models in \cite{VanMieghem2009} and \cite{SIRscoglio} respectively, to a SIRS model (with vaccination). 

In Sec. \ref{Stab}, \ste{we deal with the stability properties of the system obtained by means of the approximation. We consider the stability results in \cite{yang2017heterogeneous}, where the authors study the heterogeneous version of our node-based SIRS model. Thus, based on their results adapted to the homogeneous case, we provide the critical epidemic threshold which separates an extinction region from an endemic one in terms of the parameters of the model and the network topology.}

\ste{At this point, we focus on the global asymptotic stability (GAS) of the endemic equilibrium. 
 In
\cite{yang2017heterogeneous}, the authors 
provide a sufficient condition,  that depends on the network topology and the model parameters, for the global attractivity of the endemic equilibrium, above the epidemic threshold.}
 {However, we have not been able to find any graphs and set of parameters for which this condition is valid; even in \cite{yang2017heterogeneous} the authors do not provide numerical examples in which their condition holds. Further, 
we show that in the homogeneous setting, the sufficient condition provided in \cite{yang2017heterogeneous} is never satisfied in the case of regular graphs. }

\ste{
We underline that, to the best of our knowledge, the GAS of the endemic equilibrium for an individual-based SIRS model of our kind is still an open question, in that there are only partial results with additional strict restrictions on the model parameters (just as in \cite{yang2017heterogeneous}). The same goes for the case of a multigroup SIRS model, see e.g., \cite{lin1993global,muroya2013global}}. 

\ste{In the case of the DBMF approach, GAS of the endemic equilibrium is proved in \cite{chen2014SirsVaccNet}, under the only restriction of having a recovery rate higher than the vaccination rate. However, in a DBMF model, the assumption that all nodes with the same degree are considered stochastically equivalent simplifies the analysis, and allows to prove the GAS by means of a Lyapunov function consisting of
quadratic functions and a Volterra type function of the same kind of those used, e.g., in \cite{mena1992dynamic,shuai2013global} }.

{For these reasons, we think that it is interesting to better understand the attractiveness properties of the endemic equilibrium for our kind of model. Considering that the sufficient condition for the global attractivity in \cite{yang2017heterogeneous} does not hold for regular graphs
 in Sec. \ref{attr}, we focus on the domain of attraction of the endemic equilibrium for these specific graphs}. For this purpose, we use the notion of equitable partition \cite{godsil, QEP}. Thus, first we prove the existence of a positively invariant set for the system when a graph posses an equitable partition, then we show that, when the initial conditions belong to this set, the whole epidemic dynamics can be expressed by a reduced system of $3n$ equations, where $n < N$. 
Moreover, this invariant set contains the endemic equilibrium (besides the disease-free equilibrium) that can be computed by means of the reduced system. Since a regular graph is a special case of graph with equitable partition, we show that, when the recovery rate is higher than the vaccination rate, the aforementioned invariant set is contained in the domain of attraction for the endemic equilibrium. 
Finally, in Sec. \ref{numInv}, we provide some numerical investigations. 

\section{The Exact Model}\label{exact}

We consider a continuous-time Markovian \emph{susceptible-infected-removed-susceptible} (SIRS) \emph{model with vaccination}, on networks. Specifically, the epidemics spreads over an undirected connected graph $G=(V,E)$, where the node set $V$ represents the individuals in the population and the links between nodes are specified by the edge set $E$. The connectivity of $G$ is conveniently
expressed by the symmetric $N \times N $ adjacency matrix $A$. 
 \\
Each node can be, at time $t$, in one of the three states $S, I$, or $R$ with a certain probability. 
The state
of a node $i$, at time $t$, will be denoted by the random variable $X_i(t)$.
We assume that the infection process
is a per link Poisson process where the infection rate between a susceptible and an infected node is $\beta$. The recovery process of an infected node is poissonian too, with rate $\delta$, and once cured the individual pass from the state $I$ to $R$. We denote by $\tau=\beta/\delta$ the so-called \emph{effective infection rate}.
In a SIRS model the immunity acquired after receiving the infection is temporary (unlike the most studied SIR model). 
A recovered individual stays in the state $R$ for an exponentially distributed time with mean $1/\gamma$, before returning to the susceptible state. In addition, we include the possibility of vaccination for a healthy individual.  
We assume that each susceptible can receive vaccination at a constant rate $\sigma$ (again we have a Poisson process for vaccination), and that the vaccine is totally effective in preventing infection, although it does not provide a long-life immunity.  We do not distinguish the vaccine-induced immunity from the natural one acquired after the contraction of the disease. 
 Namely, we do not consider a vaccination state into the basic model, but the vaccinated individual pass to the state $R$. 
Thus, each individual loses the immunity either given by the vaccine or by recovering with the same rate $\gamma$. 
All the involved Poisson processes are independent. 

The state of the network $Y(t)$ at time $t$ is defined by all
possible combinations of states in which the $N$ nodes can be
at time $t$.
Let us denote the $3^N$ possible configurations that the state $Y(t)$ can assume by

$$Y_k=(X_N, \ldots, X_1),$$ 
where  $X_i \in \{S,I,R\}$
 represents the state of node $i$, and $k=0, \ldots, 3^N-1$. 
 We label the state in this way: by setting $S=0, I=1$, and $R=2$,  we can consider the vector state $Y_k$ as the ternary representation of $k$, that is $k=\sum_{i=1}^N X_i 3^{i-1}$.

The epidemics process can be described by a continuous-time
Markov chain with $3^N$ states specified by the infinitesimal
generator $Q$ with elements

\begin{equation*}\label{QSirsV}
q_{zj}=
\begin{cases}
\delta, & \qquad \text{if}\; z=j-3^{m-1} \land X_m=1; \; m=1, \ldots, N\\
\beta \sum_{i=1}^N a_{m i}\mathds{1}_{\{X_i=1\}}, &  \qquad \text{if}\; z=j-3^{m-1} \land X_m=0; \; m=1, \ldots, N\\
\gamma, & \qquad \text{if}\; z=j+ 2 \cdot 3^{m-1} \land X_m=2 ; \; m=1, \ldots, N\\
\sigma , &\qquad \text{if}\; z=j- 2 \cdot 3^{m-1}  \land X_m=0; \; m=1, \ldots, N\\
-  \sum_{i=0; i \neq z}^{3^N-1 }q_{z i}, & \qquad \text{if}\; z=j\\
0, & \qquad \text{otherwise}
\end{cases}
\end{equation*}
where, $X_m$ is the state of node $m$ in the network state $z$.\\
Let us note that  the set of all states with no infected individuals, that is those states $Y_k$, where $X_i \in \{0,2\}$, for all $i=1, \ldots, N$, forms a final class. This differentiates the SIRS$_{v}$ model from the standard SIRS one,  
where there is only one absorbing state, that is $Y_0=(X_N=S, X_{N-1}=S, \ldots,X_2=S, X_1=S)$.

%
%

Conversely, the set of states where $X_i=1$, for some $i$, forms a transient class. Standard results in Markov theory implies that the process will enter the final class in finite time, P-a.s., which is equivalent to say that the epidemic reaches the extinction (no more infected nodes) almost surely. 

Let us define the probability
state vector 

\begin{equation*}\label{psv}
v(t)=(v_0(t), \ldots, v_{3^N-1}(t)),
\end{equation*}

 with components

$$v_k(t)=\bP(Y(t)=Y_k)  .$$


The rate of
change of every network state is given by the following differential equation:
\begin{equation}\label{exactSIRS}
\frac{d v^T(t)}{dt}=  Q v^T(t),
\end{equation}

whose solution is

\begin{equation*}\label{sol}
v^T(t)= e^{Qt} v^T(0).
\end{equation*}

The system \eqref{exactSIRS} fully describes the Markov process, 
however the number
of equations increases exponentially with the number of nodes; this poses
several limitations in order to determine the 
 solutions even for small
networks. Hence, often, it is necessary to formalize models that are an
approximation of the original one, but that allow a better analytical and numerical analysis.
A direct approach for deriving an approximate model is to start from a node
level description of the underlying stochastic process, that we report in the next section. 
Then, through a mean-field type approximation (see Sec.~\ref{Secmeanf}), it is possible to obtain a reduced set of $3N$ nonlinear
differential equations describing the time-change of the state probabilities of each
node.

\subsection{Node-level Markov description of the SIRS process with vaccination}\label{nodelev}

Alternatively to the approach adopted in the previous section, we can 
describe the spreading process by a node-level approach, i.e.,
by specifying the 
probability for each node $i$ to move from a state to another, conditioned on
the network state $Y(t)$ 
\cite{PVM_GEMF}. Given a node $i$, we shall denote in the sequel $Y_{-i}(t)$ the state of all other nodes $j \ne i$ in the network.
 
{We can consider the representation for finite state Markov processes by means of all the involved Poisson processes in the model  \cite{bremaud1999markov}. }
For a susceptible individual, the process of being infected by one infected neighbor, during the interval time $(t, t+dt]$ is independent of the process of receiving infection from another neighbor. Indeed, all the infected neighbors compete with each other and the susceptible node become infected when one of the neighbors succeeds in transmitting the infection. 
Now, let us define
$\mathds{1}_{\{E\}}$ the indicator random variable (which
equals one if the condition $E$ is true, else it is zero). 
Since for the Poisson processes the probability that $q$ events occur in
a time interval ${d}t$ is of order $({d}t)^q$, we can write the probability of
having an infection for the node $i$, during the time interval $(t, t + dt]$, as
\begin{equation}\label{P1}
\bP(X_i(t+dt)=I|X_i(t)=S, Y_{-i}(t))=\beta \sum_{j=1}^N  a_{ij}\uno_{\{X_j(t) =I\}} dt + o(dt),
\end{equation}
since the sum of independent Poisson processes (i.e., the infection processes) is again a Poisson
process with 
 rate
equals to the sum of the individual rates. 
The probability of not having a transition from the infected state to the removed state, during  $(t, t+dt]$, is:
\begin{equation}\label{P2}
\bP(X_i(t+dt)=I|X_i(t)=I, Y_{-i}(t))=1 -\delta dt + o(dt).
\end{equation}
Then from \eqref{P1} and \eqref{P2}, we have
\begin{align}\label{P1P2}
\bP(X_i(t+dt)=I| Y(t)) &=\mathds{1}_{\{X_i(t)=S \}}\beta \sum_{j=1}^N  a_{ij}\mathds{1}_{\{X_j(t)=I \}}dt \\ \nonumber
&+ \mathds{1}_{\{X_i(t)=I \}}(1 - \delta dt)+o(dt).
\end{align}
By noticing that
$$\bP(X_i(t+dt)=I| Y(t))=\bE[\mathds{1}_{\{X_i(t+dt)=I \}}|Y(t)],$$
 then, if we compute the expected value of each side of \eqref{P1P2},  by the law of iterated expectation, we get
\begin{equation*}
\bE[\mathds{1}_{\{X_i(t+dt)=I \}}]= \bE \left[\mathds{1}_{\{X_i(t)=S \}} \beta \sum_{j=1}^N  a_{ij}\mathds{1}_{\{X_j(t)=I \}}\right]dt 
+ \bE \left[\mathds{1}_{\{X_i(t)=I \}} \right](1- \delta dt) + o(dt).
\end{equation*}
After dividing both members by $dt$ and letting $dt \rightarrow 0$, we have, by exploiting again the properties of the indicator random variable
\begin{equation}\label{exactI}
\frac{d \bP(X_i(t)=I )}{dt} = \beta  \sum_{j=1}^N  a_{ij} \bP(X_i(t)=S, X_j(t)=I) -\delta \bP(X_i(t)=I).
\end{equation}
The probability to be recovered, for node $i$, during the interval time $(t, t+dt]$ is
\begin{equation*}
\bP(X_i(t+dt)=R|X_i(t)=I, Y_{-i}(t))=\delta dt + o(dt).
\end{equation*}
The probability to get vaccinated during $(t, t+dt]$ is
\begin{equation*}
\bP(X_i(t+dt)=R|X_i(t)=S, Y_{-i}(t))=\sigma dt + o(dt),
\end{equation*}
and, finally, the probability that no transition from the removed state happens (that is no loss of immunity occurs) during $(t, t+dt]$, is
\begin{equation*}
\bP(X_i(t+dt)=R|X_i(t)=R, Y_{-i}(t))=1 -\gamma dt + o(dt).
\end{equation*}
Thus, proceeding as above, we have 
\begin{equation}\label{eq_rec}
\frac{d \bP(X_i(t)=R)}{dt} = \delta \bP(X_i(t)=I) + \sigma \bP(X_i(t)=S) - \gamma \bP(X_i(t)=R).
\end{equation}
With the same arguments as before, we can also discuss the variation of the probability to be in the susceptible state, to get 
\begin{equation}\label{exactS}
\frac{d \bP(X_i(t)=S)}{dt}=- \beta  \sum_{j=1}^N  a_{ij} \bP(X_i(t)=S, X_j(t)=I)  + \gamma \bP(X_i(t)=R) -\sigma \bP(X_i(t)=S). 
\end{equation}

It seems that we have described the dynamic of the system by means of $3N$ equations in the unknowns $\bP(X_i(t)=x)$, $i=1, \dots, N$, $x = S, I, R$.
Unfortunately, equations \eqref{exactI} and \eqref{exactS} are not closed since they contain the 
joint probabilities $\bP(X_i(t)=S, X_j(t)=I)$. 
We can show that it is possible to derive a system of differential equations for each joint probability, but even those are not closed, since they involve higher order joint probabilities. In the end, again, a system of $3^N$ linear
equations appears and, as for \eqref{exactSIRS}, for large values of $N$ the system is neither analytically
nor computationally tractable \cite{PVM_GEMF}. Instead, to reduce the $3^N$ state-space size, in Sec.~\ref{Secmeanf}, we adopt a closure
 approximation technique to obtain a system of $3N$ differential equations.

\subsection{Time to extinction for the SIRS model}


In this section, we 
use the dynamic described in equations \eqref{exactI}-\eqref{eq_rec}-\eqref{exactS} and discuss the average lifetime of the epidemics before its extinction
(which occurs with probability 1, since the class $Y^0 = \{Y_k : X_i \ne I,\ i=1,\dots, N\}$ is final).
Our aim is to find conditions for a
quick extinction in order to avoid a long-term epidemic persistence.

First, let us investigate the average time the epidemic is active, that is, at least one node is infected.
We focus on the SIRS model with $\sigma=0$, (\sO{although 
 the sufficient condition \eqref{threshold1} for fast extinction also applies when $\sigma >0$}), and consider $Y^0$, the set of the states with no infected nodes, which we refer to as the \emph {final set}. \\
The next proposition gives us un upper bound on $\bP \left(\sum_{i=1}^N\uno_{\{X_i(t)=I \}}>0 \right) $, that is the probability of not being in the final set $Y^0$, at time $t$. 

\begin{proposition}\label{finalSet}
\sO{Let $A$ be the adjacency matrix of the graph $G$, and $\lambda_1(A)$ its spectral radius. Then, for any initial condition $X_0=(X_1(0), \ldots, X_N(0))$, and all $t \geq 0$,  it holds:
\begin{equation*}
\bP \left(\sum_{i=1}^N\uno_{\{X_i(t)=I \}}>0 \right) \leq  \sqrt{N \sum_{i=1}^N {\uno_{\{X_i(0) = I\}}}} \exp((\beta \lambda_1(A)-\delta)t) .
\end{equation*}}
\end{proposition}

\begin{proof}

Let us consider the equation  \eqref{exactI}, by invoking the law of total probability, it can be rewritten as (\cite{van2014upper})
\begin{align*}
\frac{d \bP(X_i(t)=I )}{dt} &= \beta  \sum_{j=1}^N  a_{ij} \bP(X_j(t)=I) - \delta \bP(X_i(t)=I) \\ 
&- \beta  \sum_{j=1}^N  a_{ij} \bP(X_i(t)=I, X_j(t)=I)  - \beta  \sum_{j=1}^N  a_{ij} \bP(X_i(t)=R, X_j(t)=I), 
\end{align*}
 for $i=1, \ldots, N$. Consequently,
\begin{equation}\label{diseq_inf}
\frac{d \bP(X_i(t)=I )}{dt} \leq \beta  \sum_{j=1}^N  a_{ij} \bP(X_j(t)=I) - \delta \bP(X_i(t)=I),
\end{equation}
that written in matrix form is
\begin{equation*}
\frac{d P(t)}{dt} \leq ( \beta A -\delta\mathbb I_N) P(t),
\end{equation*}
where $P(t)=\left[ \bP(X_1(t)=I),  \ldots ,\bP(X_N(t)=I)  \right]^T$ and $\mathbb I_N$ is the identity matrix with dimension $N$.
The solution of the linear differential inequality above for the vector of infection probabilities is
\begin{equation*}
P(t) \leq \exp(t (\beta A -\delta \mathbb I_N ))P(0),
\end{equation*}
where $P(0)$ is determined by means of the initial condition $X_0$.  In the sequel we let $u$ be the all-one row vector.
We notice that, for any $i=1,\dots, N$,
$$ \bP(X_i(t)=I)= \sum_{Y_k: X_i = I}  \bP(Y(t)=Y_k);$$
further,
\begin{align*}
\bP(Y(t) \not\in Y^0) =& \sum_{Y_k \not\in Y^0} \bP(Y(t)=Y_k)
\le \sum_{i=1}^N \sum_{ Y_k: X_i = I} \bP(Y(t)=Y_k)
\\
= & \sum_{i=1}^N \bP(X_i(t)=I)
\le \sum_{i=1}^N \big(\exp(t (\beta A -\delta \mathbb I_N ))P(0) \big)_i
\\
= & u \cdot \exp(t (\beta A -\delta \mathbb I_N ))P(0).
\end{align*}
By invoking the Cauchy-Schwarz inequality and considering that the matrix $A$ is symmetric, we obtain (see \cite[Thm 8.2]{draief2010epidemics})
\begin{align*}
\bP  \left(\sum_{i=1}^N\uno_{\{X_i(t)=I \}}>0 \right)  & \leq ||u||_2 \exp((\beta \lambda_1(A) -\delta )t)||P(0)||_2 \\
& = \sqrt{N \sum_{i=1}^N \uno_{\{X_i(0) = I\}}} \exp((\beta \lambda_1(A)-\delta)t)
\end{align*} 
as claimed.
\end{proof}

\begin{corollary}\label{mean_timeFS}
\sO{
Let $\tau^{FS}$ denote the hitting time to the final set $Y^0$. Then, under the condition
\begin{equation}\label{threshold1}
\frac{\beta}{\delta} < \frac{1}{\lambda_1(A)}
\end{equation}
it holds that 
\begin{equation*}
\bE(\tau^{FS}) \leq \frac{\log(N) +1}{\delta - \beta \lambda_1(A)}.
\end{equation*}}
\end{corollary}

\begin{proof}
Following the proof of \cite[Cor. 8.6]{draief2010epidemics} we have
\begin{equation*}\label{tfs}
\bE(\tau^{FS})= \int_0^\infty \bP( \tau^{FS}>t) dt 
=  \int_0^\infty  \bP \left(\sum_{i=1}^N\uno_{\{X_i(t)=I \}}>0 \right)  dt 
 \leq  \int_0^\infty \min \{1, N \exp(-(\delta- \beta \lambda_1(A))t)\} dt.
\end{equation*}
Since $N \exp(-(\delta- \beta \lambda_1(A))t) < 1$ when $t  > \log(N)/(\delta -\beta \lambda_1(A))=t^* $ 
 we can split the intervals of integration in $[0, t^*]$ and $[t^*, \infty]$, obtaining that

$$\bE(\tau^{FS}) \leq t^* + \frac{N}{\delta- \beta \lambda_1(A)} \exp(-(\delta- \beta \lambda_1(A))t^*)=\frac{\log(N) +1}{\delta - \beta \lambda_1(A)}.$$
\end{proof}

The above result states that if we consider a sequence of graphs $G_N$ on $N$ nodes, 
for instance regular graphs with fixed degree $k$ (notice that they share the same spectral radius $\lambda_1(A_N) = k$),
then the condition
$\displaystyle \delta - \beta \lambda_1(A_N) \ge c > 0$, for some constant $c$, 
implies that the expected time to the infection eradication grows at most logarithmically in $N$. 
In this setting, for large $N$, by using Markov's inequality we have that the time to eradication is of order $(\log(N))^\alpha$ with high probability, for any $\alpha > 1$.

{The result in Proposition \ref{finalSet} implies that the condition \eqref{threshold1} is sufficient for fast extinction}.  
This coincides with what is known for the SIS model in \cite[Thm 8.2]{draief2010epidemics} where the bound is over the probability that at time $t$ the process has not yet reached the absorbing state (the overall-healthy state).
%


\sO{\textit{Time to absorbing state}.} In the previous section we have considered the probability of the persistence of the epidemics (meaning that at least one infected node remains in the network) and the mean time to hit the final set, where there are no more infectious nodes. \sO{Now, instead we want to consider the probability of no absorption for the SIRS model ($\sigma=0$), that is the probability that the process is not in the zero state, where all nodes are susceptible}.
%
%
%
\begin{proposition}
\sO{Under the same hypothesis of the Prop. \ref{finalSet}, and assuming  that $- \gamma$ does not belong to the spectrum of $\beta A - \delta \; \mathbb I_N$,  it holds that}

\sO{\begin{equation*}
\bP \left(\sum_{i=1}^N X_i(t) > 0 \right) \leq C \;\sqrt{N \sum_{i=1}^N \uno_{\{X_i(0) = I \lor R\}}} \; \exp (\max\{\beta \lambda_1(A)-\delta, -\gamma\} \;t),
\end{equation*}}
\sO{where $C$ is a positive constant that depends on the adjacency matrix $A$, and on the parameters $\beta, \delta, \gamma$}.

  \end{proposition}

\begin{proof}

By considering equations  \eqref{eq_rec} and \eqref{diseq_inf} we can write

\begin{equation*}
\frac{d \ov P(t)}{dt} \leq  \ov{A} \; \overline{P}(t),
\end{equation*}

where $\ov P(t)=\left[ P_I(t), P_R(t)\right]^T$, with

\begin{equation*}
P_I(t)=\left[ \bP(X_1(t)=I),  \ldots, \bP(X_N(t)=I)  \right]^T, \qquad P_R(t)=\left[ \bP(X_1(t)=R), \ldots, \bP(X_N(t)=R)  \right]^T, 
\end{equation*}

and 
\[
\ov A=
\left[
\begin{array}{c c}
\beta A - \delta \; \mathbb I_N &  0\\

\delta \; \mathbb I_N & - \gamma \; \mathbb I_N
\end{array}
\right].
\]

Thus,

$$\ov P(t) \leq \exp(t \ov A)\ov P(0),$$

with $\ov P(0)=[P_I(0), P_R(0)]^T$, which is determined by the initial condition $X_0$. Consequently
\begin{equation*}
\bP \left(\sum_{i=1}^N X_i(t) >0 \right) 
\leq u \exp(t \ov A  )\ov P(0).
\end{equation*}

By invoking Cauchy-Schwarz inequality we arrive at

$$\bP \left(\sum_{i=1}^N X_i(t) > 0 \right) \leq  ||u||_2 ||\exp(t \ov A) ||_2 ||\ov P(0)||_2.$$

\sO{The matrix $\ov A$ is diagonalizable 
if $- \gamma$ does not belong to the spectrum of $\beta A - \delta \; \mathbb I_N$. Indeed, it easy to see that under this hypothesis a basis of eigenvectors of $\ov A$ can be found}. Thus, we have that $||\exp(t \ov A) ||_2= ||M \exp(D t) M^{-1}||_2$, where
$D$ is the diagonal matrix containing the eigenvalues of $\ov A$ and $M$ the matrix
containing the corresponding eigenvectors. Finally, we have 

\begin{equation*}
\bP \left(\sum_{i=1}^N X_i(t) > 0 \right) \leq  C \;\sqrt{N \sum_{i=1}^N  \uno_{\{X_i(0) = I \lor R}\} } \; \exp (\lambda_1(D) t)  
\end{equation*}

where $$\lambda_1(D)=\max\{\beta \lambda_1(A)-\delta, -\gamma\}$$ is the maximum eigenvalue of the matrix $\ov A$, and $C= ||M||_2||M^{-1}||_2$.
\end{proof}


\emph{Numerical investigations.} We investigate numerically the role of the immunity-loss parameter $\gamma$ on the mean fraction of infected nodes, in both the exact SIRS and SIRS$_{v}$ models.
We consider the averaged $ 10^3$ sample paths resulting from a discrete event simulation of
the stochastic processes. The discrete event simulation is based on the generation of
independent Poisson processes for the infection of healthy nodes, the recovery
of infected, and for the loss of immunity of the removed, and for the vaccination of susceptible in the SIRS$_v$, i.e., when $\sigma >0$.

We can see that $\gamma$ influences the dynamics of the average fraction of infected nodes (the prevalence). Specifically, in Fig.~\ref{fig:SIM_SIRS_1} a) we show the behavior of the prevalence for the exact SIRS as function of time and $\gamma$. 
 We consider the complete graph with $N=50$ and fixed values of $\beta$ and $\delta$ for which the condition \eqref{threshold1} does not hold.
 \sO{We observe that for some low values of  $\gamma$ the average fraction of infected nodes
decays towards zero in a quite short time window. As $\gamma$
 grows the time to extinction tends to increase, and after a certain critical value of $\gamma$ the prevalence tends to stabilize around a positive quantity for long time 
 (resembling the behavior of the mean-field model that above the threshold reaches the positive equilibrium point (Sec. \ref{Secmeanf})).  The same beahvior can be observed
 for the exact SIRS$_v$ in b). Thus, we are led to assert that \emph{the value of $\gamma$ influences the time to extinction of the exact models}.}

 \begin{figure}[t]
 \centering
    \begin{minipage}{0.45\textwidth}
    \includegraphics[width=\textwidth,height=5.3cm]{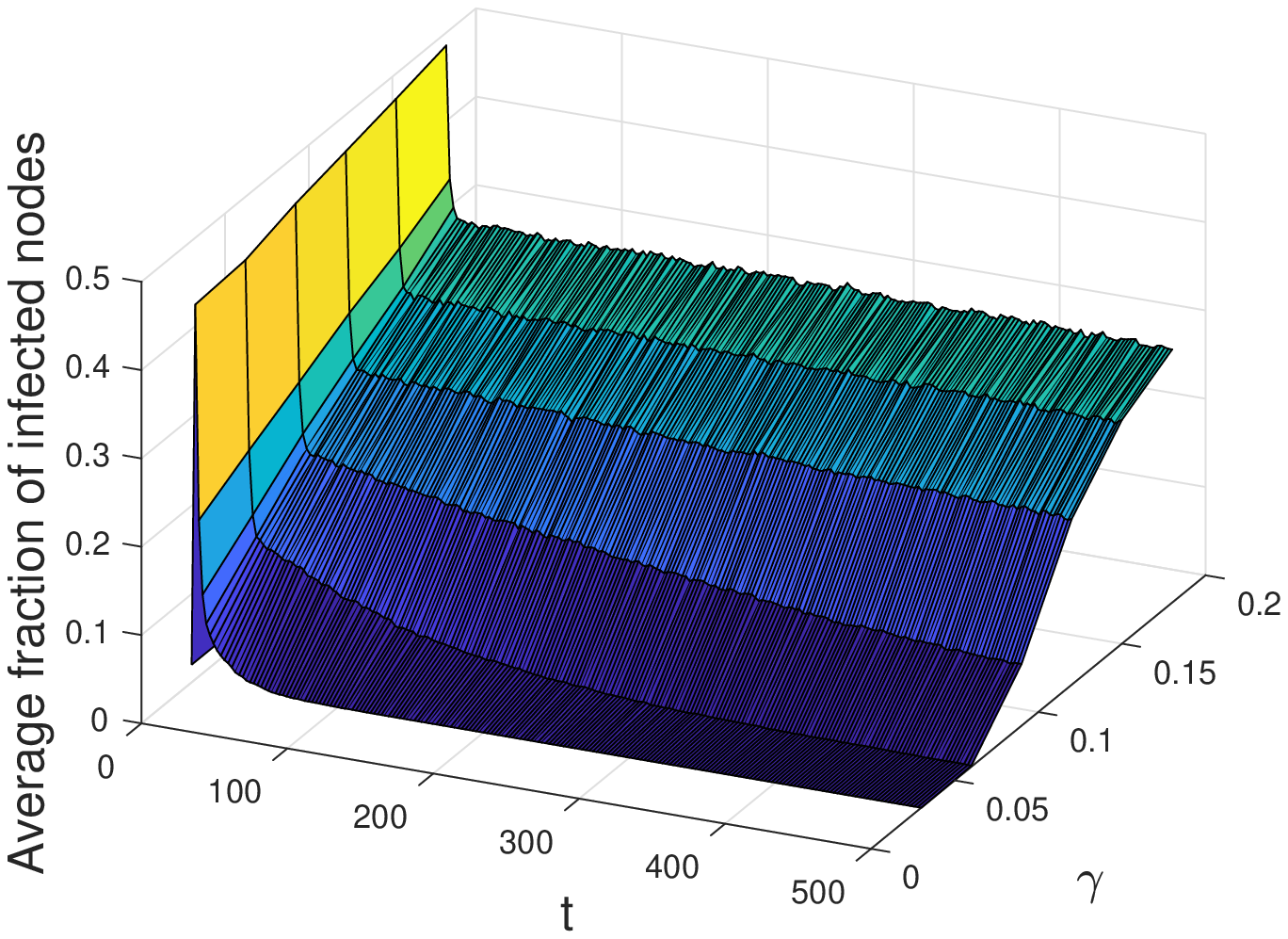}\put(-168,140){a)}
    \end{minipage}
   \hskip4mm
   \begin{minipage}{0.45\textwidth}
    \includegraphics[width=\textwidth,height=5.4cm]{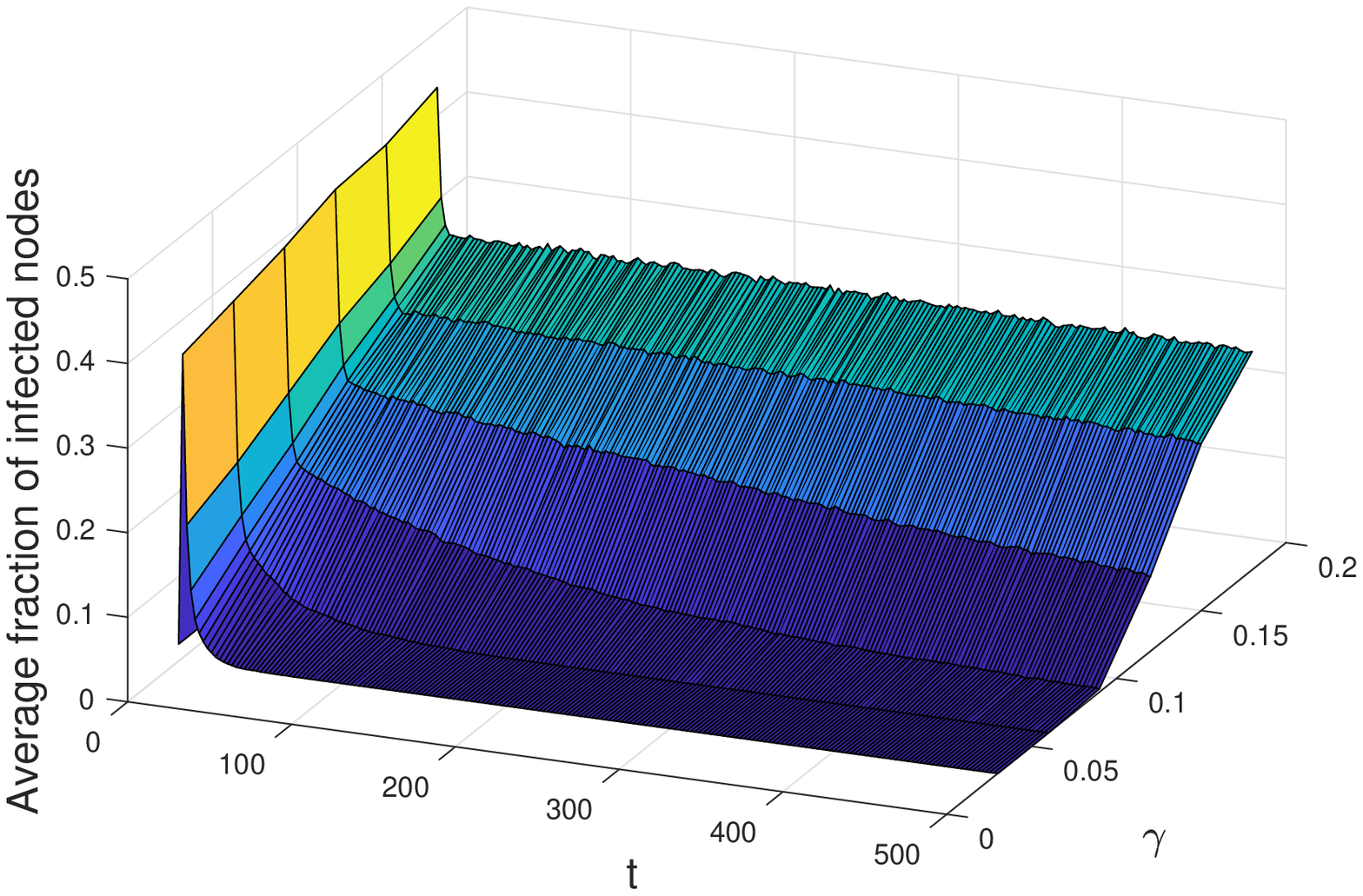}\put(-162,140){b)}
\end{minipage}
  \caption{ \ste{SIRS and SIRS$_{v}$ mean fraction of infected nodes as function of time and $\gamma$, obtained by averaging  $10^3$ simulated sample paths}, for a complete graph with $N=50$. a) SIRS model, $\beta=0.25$, $\delta=0.4$. b) SIRS$_v$ model, $\beta=0.25$, $\delta=0.4$, $\sigma=0.45$. At time 0 there is one infected node.}
\label{fig:SIM_SIRS_1}
  \end{figure}

\section{Mean-field approximation}\label{Secmeanf}

Let us come back to the node-level description for the Markov model (Sec.~\ref{nodelev}). As we pointed out equations \eqref{exactI} and \eqref{exactS} are not closed since they contain the joint probabilities $\bP(X_i(t)=S, X_j(t)=I)$.
We can "close" the equations providing an approximation for
the joint probabilities in terms of the marginal probabilities, assuming the independence between the dynamic states of two neighbors, the so-called first-order mean-field type approximation \cite{ PVM_GEMF}.  
Thus, let $(i,j) \in E$, we assume

\begin{equation}\label{meanf}
\bP(X_i(t)=S, X_j(t)=I)=\bP(X_i(t)=S)\bP(X_j(t)=I).
\end{equation}
Let us define the state probabilities of individual $i$, at time $t$, as 

\begin{equation*}
S_i(t)=\bP(X_i(t)=S), \qquad  I_i(t)=\bP(X_i(t)=I), \qquad R_i(t)=\bP(X_i(t)=R).
\end{equation*}
 Then, by means of the assumption \eqref{meanf}, we have the following \emph{mean-field} equations for the SIRS$_v$ model 

\begin{align}\label{3dim}
\frac{dS_i(t)}{d t} &= -S_i(t)\sum_{j=1}^N \beta a_{ij} I_j(t)+ \gamma R_i(t)- \sigma S_i(t) \nonumber\\
\frac{d I_i(t)}{dt} &= S_i(t)\sum_{j=1}^N \beta a_{ij} I_j(t) -\delta I_i(t)\\
\frac{d R_i(t)}{dt} &= \delta I_i(t) - \gamma R_i(t) + \sigma S_i(t), \nonumber
\end{align}
for $i=1, \ldots, N$, with initial conditions

$$((S_1(0), \ldots, S_N(0), I_1(0), \ldots, I_N(0), R_1(0), \ldots, R_N(0)) \in \tilde \Gamma,$$

$$\tilde \Gamma =\{(S_1, \ldots, S_N, I_1, \ldots, I_N,  R_1, \ldots, R_N ) \in \sO{\mathbb R^{3N}_+}| S_i + I_i +R_i=1, i=1,2, \ldots, N\}$$
where \sO{$\mathbb R^{3N}_+$ is the non-negative orthant of $\mathbb{R}^{3N}$}.
Since $S_i(t) + I_i(t) + R_i(t)=1$, we can omit the equation for the probability of being in the susceptible state and obtain  
\begin{align}\label{2dim}
\frac{d I_i(t)}{dt} &= (1-I_i(t)-R_i(t))\sum_{j=1}^N \beta a_{ij} I_j(t) -\delta I_i(t),   \nonumber \\
\frac{dR_i(t)}{d t} &=(\delta-\sigma) I_i(t) -(\gamma + \sigma) R_i(t) +\sigma,  
\end{align}

for  $i=1, \ldots, N$, with initial conditions

$$( I_1(0), \ldots, I_N(0),R_1(0), \ldots, R_N(0)) \in \Gamma,$$

where

$$\Gamma =\{( I_1, \ldots, I_N, R_1, \ldots, R_N) \in  \sO{\mathbb{R}^{2N}_+}| I_i + R_i \leq 1, i=1,2, \ldots, N\}.$$

The region $\tilde \Gamma$ and $\Gamma$ are positively invariant for the system \eqref{3dim} and \eqref{2dim}, respectively (see \cite{yang2017heterogeneous}).\\\\

\sO{As discussed for the SIS and SIR model in literature, we conjecture that also for the SIRS model the following inequality holds}

\sO{\begin{equation}\label{correl}
\bP(X_i(t)=S, X_j(t)=I) \leq \bP(X_i(t)=S)\bP(X_j(t)=I),
\end{equation} }
\sO{ that is}
\begin{equation*}\label{pcond}
\bP(X_i(t)=S|X_j(t)=I) \leq  \bP(X_i(t)=S),
\end{equation*}
for all $t \geq 0$.
\sO{ The intuitive idea behind this is that an infected neighbor does not increase the probability
of an individual to remain susceptible \cite{SIRscoglio}. 
A first rigorous proof of the positive correlation between infection states was provided in \cite{donnelly1993correlation} for the SIS Markov model and for a general
(non-Markov) SIR model, and later again proved for the Markovian SIS in \cite{cator2014nodal} (see also the discussion in \cite{cator2018reply})}.
 
\sO{If \eqref{correl} holds for our SIRS model, we would have that the derivative of the infection probability in \eqref{3dim} is always overestimated
as a consequence of the independence assumed in \eqref{meanf}. 
 Thus,
 the probability of infection for each node in the approximated model would provide an upper bound of the exact infection probability in the Markov model. This seems also to be confirmed by the simulations reported in Section \ref{numInv}, where we compare the exact model with the approximated one. 
Hence, from
a practical point of view, to prevent epidemics in a network,
the mean-field model would put us always on the safe side, as provided for other types of epidemic models \cite{VanMieghem2009,SIRscoglio, cator2018reply}.}

\section{Stability analysis}\label{Stab}

The \emph{disease free equilibrium} (DFE) of the system \eqref{2dim} is given by the vector $P_0=(I^0_1, \ldots I^0_N, R^0_1, \ldots, R^0_N)$, where

\begin{equation*}
I^0_i=0,  \qquad \text{and} \qquad R^0_i=\frac{\sigma}{\gamma+\sigma}, \qquad  i= 1, \ldots, N. 
\end{equation*}
Let us note that for the SIRS model without vaccination, i.e. $\sigma=0$, $R^0_i=0$ and $I^0_i=0$, for $i= 1, \ldots, N$. \\
The positive constant solution, i.e., the \emph{endemic equilibrium} $P^*$, for the SIRS$_v$  model \eqref{2dim} has the following components

\begin{align*}
I_i^*&= \frac{1}{\delta}\frac{\gamma \sum_{j=1}^N \beta a_{ij}I^*_j}{\sum_{j=1}^N \beta a_{ij}I^*_j(1+ \gamma/\delta)+(\gamma + \sigma)},\\
R_i^*&= 1- I_i^*- \frac{\gamma}{\sum_{j=1}^N \beta a_{ij}I^*_j(1+ \gamma/\delta)+(\gamma + \sigma)},
\end{align*}
for $i=1, \ldots, N$. Summing $I_i^*$  over all nodes, and divided by $N$, we obtain the average fraction of infected nodes in the steady state, $\bar{I}^*$.

\ste{Below, we recall some stability results from \cite{yang2017heterogeneous}, where the authors consider a heterogeneous version of \eqref{2dim}, adapting them to the homogeneous case.}

\begin{theorem}\label{eq0}

Let us consider the system \eqref{2dim} and let $D=\frac{\gamma}{\gamma + \sigma}\beta A - \delta \mathbb I_N $, whose maximum eigenvalue is 

\begin{equation*}\label{maxD}
\lambda_{1}(D)=\frac{\gamma}{\gamma + \sigma}\beta \lambda_{max}(A) - \delta.
\end{equation*}
The following statements hold
\begin{itemize}
\item[a)]If $\tau \leq \frac{\gamma + \sigma}{\gamma}\frac{1}{\lambda_1(A)}$ the disease free equilibrium $P_0$ is globally asymptotically stable.  
$P_0$ is the unique equilibrium of the system \eqref{2dim} on the boundary of $\Gamma$.
\item[b)]If $\tau > \frac{\gamma + \sigma}{\gamma}\frac{1}{\lambda_1(A)}$, $P_0$ is a saddle point, the system \eqref{2dim} is uniformly persistent and it has a unique positive constant solution $P^*$ in  \sO{$\mathring{\Gamma} $}. Moreover if $\delta \geq \sigma $, $P^*$ is asymptotically stable.
\end{itemize}
\end{theorem}

Thus, for the SIRS$_v$ model \eqref{2dim}, the critical threshold separating the region of extinction from the persistent one is 

\begin{equation}\label{thrSIRSV}
\tau_{c;SIRS_{v}}^{(1)}=\frac{\gamma + \sigma}{\gamma}\frac{1}{\lambda_1(A)}.
\end{equation}

In \cite{yang2017heterogeneous}, the authors give also sufficient conditions for the global stability of the endemic equilibrium in $\mathring{\Gamma}$. Precisely, in the homogeneous setting, we have:

\begin{theorem}\label{glob}
Let $\tau > \frac{\gamma + \sigma}{\gamma}\frac{1}{\lambda_1(A)}$. Then, $P^*$ is 
\sO{globally} asymptotically stable in $\mathring{\Gamma}$, if one of the following two conditions hold:

\begin{itemize}
\item[a)]  $\delta > \sigma$, \; and $\lambda_1(A) < \frac{1}{\beta} \cdot \displaystyle \min_i \left\{ \frac{\delta I^*_i}{(S_i^*)^2} \right\} \cdot   \displaystyle \min_i \left\{ \frac{S_i^*}{1-S_i^*}\right\},$
\item[b)]$ \delta= \sigma$.
\end{itemize}

\end{theorem}


\sO{
Let us note that the condition in a), regarding the maximum eigenvalue of $A$, {might be difficult to satisfy: we have not been able to find graphs and parameters for which this condition is valid, and even in \cite{yang2017heterogeneous} the authors do not provide numerical examples in which the condition holds.}
 For example, in Sec. \ref{subsec_reg}, we shall prove that, in the homogeneous setting, for the case of regular graphs condition a) is never satisfied, hence it cannot be used for verifying the global attractivity of the endemic equilibrium. Thus, in the next section, we shall investigate the attractivity of the endemic equilibrium in this specific case.
 Specifically, for dynamics over a regular graph, we find an invariant subset of $\tilde{\Gamma}$ (and, consequently of $\Gamma$), and we prove that, above the threshold and under the condition $\delta > \sigma$,
this subset is included in the domain of attraction of the endemic equilibrium.
To prove this we pass through the theory of equitable partitions and we shall see how in this particular case the equilibrium points can be computed by a reduced system.}

\subsection{\sO{Attractivity of the the endemic equilibrium: regular graphs}}\label{attr}

In this section we dwell on the graph-theoretical
notion of equitable partition \cite{Schwenk,godsil,EvolDelio}.
A network with an equitable partition of its node set posses certain structural regularity
of the graph connectivity. Based on this, we shall analyse the domain of attraction of the endemic equilibrium in the case of regular graphs that can be seen as a graph with an equitable partition. 

\subsubsection{Equitable partitions}
In the following, we report the definition of equitable partition \cite{Schwenk}.
\begin{definition}\label{def:eqpart}
Let $G=(V,E)$ be an undirected graph. The partition $\pi=\left\{V_1,...,V_n\right\}$ of the node set $V$ is called \emph{equitable} if  
 for all $i,j \in \left\{1, \dots ,n \right\}$, there is an integer $d_{ij}$ such that 
\begin{equation*}
d_{ij}=\mbox{\rm deg}(v,V_j):=\# \left\{e \in E : e=\left\{v,w\right\}, w \in V_j \right\}.
\end{equation*} 
independently of $v \in V_i$.
\end{definition}

An equitable partition generates the \emph{quotient graph} $G/\pi$, which is a
\emph{multigraph} with the cells $V_1,...,V_n$ as nodes and $d_{ij}$ edges between $V_i$ and $V_j$. 
For simplicity, one can identify $G/\pi$ in a (simple) 
graph having the same node set, and where an edge exists between
$V_i$ and $V_j$ if at least one exists in the original multigraph \cite{QEP}.


This partition of the node set can be adopted for representing a population divided in communities, a framework that captures some of the most salient structural inhomogeneities in contact patterns in many
applied contexts \cite{pellis2015seven, ottaviano2018optimal}. For an overview of the use of equitable partitions, from a theoretical and practical point of view, see e.g., \cite{QEP,ottaviano2018optimal,ottaviano2019community,neuberger2020}. 
 One can identify the set of all nodes in $V_i$ as the $i$-th \emph{community} of the whole population.
In particular, each $V_i$ induces a subgraph, $G_i$, of $G$ that is necessarily regular.
 Hereafter, as in \cite{QEP}, we consider two infection rates: the \emph{intra-community} infection rate $\beta$ for infecting individuals in the same community and the \emph{inter-community} infection rate $\eps\beta$  i.e., the rate at which individuals
among different communities get infected.  We assume  $0<\eps < 1$, the customary physical interpretation is that infection across communities occur at a much smaller rate. Clearly the model can be extended to the case $\eps > 1$.

\sO{In the case of two different infection rates,  we replace the unweighted adjacency matrix in the system \eqref{3dim} with its weighted version, incorporating the parameter $\eps$ (see \cite[Example 3.1]{QEP}). 
Interestingly, the spectral radius of the smaller quotient graph (that is of the \emph{quotient matrix} related to the quotient graph (see \cite{QEP})) is equal to the spectral radius of the matrix $A$ (see \cite[Prop 3.3]{QEP}).}

 \sO{In \cite{QEP}, the authors show that it is possible to reduce the number of equations
representing the time-change of infection probabilities 
 when all nodes belonging to the same cell have the same initial conditions.
After proving the existence of a positively invariant set for the original system of $N$ 
differential equations, they show that the endemic equilibrium 
belongs to this invariant set and that it can be computed by the reduced system
of $n < N$ equations. In the following, we want to prove the same for the case of the SIRS model (with vaccination).}

Let us consider the average value of the state probabilities at time $t$ of nodes in $G_h$, 
\begin{equation*}
\ov{S}_h(t)= \frac{1}{k_h}\sum_{i \in G_h} S_i(t),  \qquad  \ov{I}_h(t)= \frac{1}{k_h}\sum_{i \in G_h} I_i(t), \qquad \ov{R}_h(t)= \frac{1}{k_h}\sum_{i \in G_h} R_i(t),  
\end{equation*}
where $k_h$ is the cardinality of $G_h$,  $h=1, \ldots, n$. Then, it holds

\begin{theorem}\label{thmRed}
Let $G=(V,E)$ be an undirected graph and $\pi = \{V_h,\ h = 1, \dots, n\}$ be an equitable partition of the node set $V$. Let $G_h$ be the subgraph of $G=(V,E)$ induced by the cell $V_h$. Let $Y=(S_1, \ldots, S_N, I_1, \ldots, I_N, R_1, \ldots, R_N) \in \tilde \Gamma$. Then, the subset of $\tilde \Gamma$


\begin{align}\label{inv_set}
\tilde \Omega =  \{ Y \in \tilde{\Gamma}| & S_i = \ov S_h,  I_i=\ov I_h, R_i=\ov R_h,\ \forall\, i \in G_h,\ h=1, \dots, n \} 
\end{align}
is positively invariant for the system \eqref{3dim}.
\end{theorem}

\begin{proof}

 From \eqref{3dim}, we have for all $i \in G_h,  \quad h=1, \ldots, n$

\begin{align}\label{ridS1}
\frac{d (S_i-\ov S_h)}{dt}  
& =- \beta \left[ S_i \sum_{z=1}^N  a_{iz} I_z  - \frac{1}{k_h}  \sum_{r \in G_h}\sum_{z=1}^N S_r   a_{rz}I_z    \right] + \gamma (R_i-\ov R_h)\\  \nonumber
& - \sigma (S_i- \ov S_h)  \nonumber
\end{align}
\begin{align}\label{ridI1}
\frac{d (I_i-\ov I_h)}{dt} &= \beta \left[ S_i \sum_{z=1}^N  a_{iz} I_z  - \frac{1}{k_h}  \sum_{r \in G_h}\sum_{z=1}^N S_r   a_{rz}I_z    \right] -  \delta (I_i - \ov I_h), 
\end{align}
\begin{align}\label{ridR1}
\frac{d (R_i-\ov R_h)}{dt} &=   \delta (I_i - \ov I_h) - \gamma (R_i-\ov R_h) + \sigma (S_i - \ov S_h),   \hspace{0.2cm} 
\end{align}

Now, let us focus on the nonlinear part in \eqref{ridS1} (and in \eqref{ridI1}). We have

\begin{align}\label{ridS2}
&- \beta \left[ S_i \sum_{z=1}^N  a_{iz} I_z - \frac{1}{k_h} \sum_{r \in G_h}\sum_{z=1}^N S_r   a_{rz}I_z   \right]   
 = - \beta \left[ \sum_{m=1}^n \sum_{z \in G_m} a_{i z} S_i I_z  -  \frac{1}{k_h}   \sum_{r \in G_h}  \sum_{m=1}^n \sum_{z \in G_m} a_{rz}  S_r I_z  \right]\\  \nonumber
&= -\beta  \frac{1}{k_h} \sum_{r \in G_h}  \left[  \sum_{m=1}^n \sum_{z \in G_m}a_{i z} S_i I_z -  \sum_{m=1}^n \sum_{z \in G_m} a_{rz}  S_r  I_z    \right]  
= -\beta  \frac{1}{k_h} \sum_{r \in G_h}  \sum_{m=1}^n \sum_{z \in G_m} (a_{i z} S_i - a_{rz} S_r) I_z  \\  \nonumber
&= -\beta  \frac{1}{k_h} \left[ \sum_{r \in G_h}  \sum_{m=1}^n \sum_{z \in G_m}  (a_{iz}(S_i -\ov S_h)- a_{rz}(S_r -\ov S_h) ) I_z+  (a_{iz}-a_{rz}) \ov S_h I_z \right] \\ \nonumber 
&=-\beta  \frac{1}{k_h} \left[ \sum_{r \in G_h}  \sum_{m=1}^n \sum_{z \in G_m} (a_{iz}(S_i -\ov S_h)- a_{rz}(S_r -\ov S_h))(I_z - \ov I_m) \right. \\ \nonumber
& +  \sum_{r \in G_h}  \sum_{m=1}^n \sum_{z \in G_m} (a_{iz}(S_i -\ov S_h)- a_{rz}(S_r -\ov S_h)) \ov I_m    \\ \nonumber
& + \left. \sum_{r \in G_h}  \sum_{m=1}^n \sum_{z \in G_m} (a_{iz}-a_{rz})\ov S_h (I_z- \ov I_m) + \sum_{r \in G_h}  \sum_{m=1}^n \sum_{z \in G_m} (a_{iz}-a_{rz}) \ov S_h \ov I_m \right] \nonumber
\end{align}
 
Now, from the last equation in \eqref{ridS2}

\begin{equation}\label{zero}
 \frac{1}{k_h}  \sum_{r \in G_h}  \sum_{m=1}^n \sum_{z \in G_m} (a_{iz} -a_{rz}) \ov S_h \ov I_m=  \frac{1}{k_h} \ov S_h \sum_{r \in C_h}  \sum_{m=1}^n \ov I_m \sum_{z \in G_m} (a_{iz} -a_{rz}),
\end{equation}

Then, since $\forall \; i,r \in G_h$ and $\forall \; m\in \{  1, \ldots, n \}$, $\sum_{z \in G_m} (a_{iz} -a_{rz})=0$, we have that \eqref{zero} is equal to zero. 
Finally, from \eqref{ridS1} and \eqref{ridS2}, we come to have

\begin{align}\label{meanS}
\frac{d (S_i-\ov S_h)}{dt} &=  -\beta  \frac{1}{k_h} \left[ \sum_{r \in G_h}  \sum_{m=1}^n \sum_{z \in G_m} (a_{iz}(S_i -\ov S_h)- a_{rz}(S_r -\ov S_h))(I_z - \ov I_m) \right. \\ \nonumber
& + \left.  \sum_{r \in G_h}  \sum_{m=1}^n \sum_{z \in G_m} (a_{iz}(S_i -\ov S_h)- a_{rz}(S_r -\ov S_h)) \ov I_m + \right. \\ \nonumber
& \left. \sum_{r \in G_h}  \sum_{m=1}^n \sum_{z \in G_m} (a_{iz}-a_{rz})\ov S_h (I_z- \ov I_m)  \right] - \sigma (S_i - \ov S_h) + \gamma (R_i -\ov R_h),  
\end{align} 

$  \forall i \in G_h$,  $ h=1, \ldots, n$.
Similarly,

\begin{align}\label{meanI}
\frac{d (I_i-\ov I_h)}{dt} &=  \beta  \frac{1}{k_h} \left[ \sum_{r \in G_h}  \sum_{m=1}^n \sum_{z \in G_m} (a_{iz}(S_i -\ov S_h)- a_{rz}(S_r -\ov S_h))(I_z - \ov I_m) \right. \\ \nonumber
& + \left.  \sum_{r \in G_h}  \sum_{m=1}^n \sum_{z \in G_m} (a_{iz}(S_i -\ov S_h)- a_{rz}(S_r -\ov S_h)) \ov I_m + \right. \\ \nonumber
& \left. \sum_{r \in G_h}  \sum_{m=1}^n \sum_{z \in G_m} (a_{iz}-a_{rz})\ov S_h (I_z- \ov I_m)  \right] - \delta (I_i - \ov I_h), 
\end{align}

$  \forall i \in G_h$,  $ h=1, \ldots, n$.
%

Now, let us denote by $g(t)$ the solution of the system $\mathcal{G}$, with equations \eqref{ridS1}, \eqref{ridI1}, \eqref{ridR1}, where $g: \mathbb{R} \rightarrow \mathbb{R}^{3N}$ and consider the case where
\begin{equation}\label{SameIniz}
S_i(0)- \ov{S}_h(0)=0,   \qquad
I_i(0)- \ov{I}_h(0)=0, \qquad
R_i(0)-\ov  R_h(0)=0   \qquad \forall i \in G_h,
\end{equation}
that means, $S_i(0)=S_r(0)$,  $I_i(0)=I_r(0)$, $R_i(0)=R_r(0)$, for all $i, r \in G_h$, $ h=1, \ldots, n$.
Then, from \eqref{meanS}, \eqref{meanI}, \eqref{ridR1} we can easily see that the identically zero function $g \equiv 0$  is the unique solution of the system $\mathcal{G}$, with initial conditions \eqref{SameIniz}.
Indeed, $g \equiv 0$, means that for all $t \geq 0$,
\begin{equation}\label{sametrajec} 
S_i(t)=S_r(t), \qquad I_i(t)=I_r(t),  \qquad R_i(t)=R_r(t),  \qquad \forall i, r \in G_h,
\end{equation}
$h=1, \ldots, n.$ Moreover, the vector \sO{with components as in \eqref{sametrajec} is a solution of \eqref{3dim} and it is unique in $\tilde{\Gamma}$, with respect to the initial
conditions \eqref{SameIniz}, hence $g \equiv 0$ is
the unique solution of  $\mathcal{G}$. Thus, we have that $\tilde \Omega$ is positively invariant for system \eqref{3dim}}.
\end{proof}

Thus, under the hypothesis in Thm. \ref{thmRed}, considering initial conditions in $\tilde \Omega$, we can reduce the original system \eqref{3dim} of $3N$ differential equations and
describe the time-change of the state probabilities
by a system of $3n$ equations. The same argument can be applied to system \eqref{2dim}. 
Specifically, we have

\begin{align}\label{RedSIRS}
\frac{d \ov S_h}{dt} &= -\beta \ov S_h \sum_{m=1 ; m \neq h}^n \eps d_{h m} \ov I_m - \beta \ov S_h d_{h} \ov I_h + \gamma \ov R_h - \sigma \ov S_h,\\ \nonumber
\frac{d \ov I_h}{dt} &= \beta \ov S_h \sum_{m=1 ; m \neq h}^n \eps d_{h m} \ov I_m + \beta \ov S_h  d_{h } \ov I_h -\delta \ov I_h, \\  \nonumber
\frac{d \ov R_h}{dt} &= \delta \ov I_h- \gamma \ov R_h + \sigma \ov S_h, \hspace{3  cm} h=1, \ldots, n
\end{align}
where $d_h$ is the internal degree of $G_h$.

\begin{remark}\label{remSteady}
\sO{From the uniqueness argument in Thm. \ref{eq0} b), it is immediate to deduce that when $G$ is a graph with an equitable partition the endemic equilibrium of the system \eqref{3dim} must belong to $\tilde{\Omega}\cap \mathring {\tilde{\Gamma}}$. Thus, it can be computed by means of the reduced system \eqref{RedSIRS}}.
\end{remark}

\subsubsection{Regular Graphs}\label{subsec_reg}

\ste{In this section we investigate the domain of attraction of the endemic equilibrium for the case of regular graphs. Indeed, we can see that the sufficient condition in Thm. \ref{glob} ensuring the global attractiveness of the endemic equilibrium, above the threshold \eqref{thrSIRSV}, }

\begin{equation}\label{cond_maxD}
\lambda_1(A) < \frac{1}{\beta} \cdot  \displaystyle \min_i \left\{ \frac{\delta I^*_i}{(S_i^*)^2} \right\} \cdot  \displaystyle  \min_i \left\{ \frac{S_i^*}{1-S_i^*}\right\},
\end{equation}
\ste{is never satisfied in the case of regular graphs. 
We can prove this fact by means of the results obtained above for the equitable partitions, since
 regular graphs, where all nodes have the same degree $d_G$, can be considered as having an equitable partition with a single cell.  
Then, from Remark \ref{remSteady}, we have that $S^*_i=S^*_j=S^*$ and $I^*_i=I^*_j=I^*$ for all $i,j=1, \ldots, N$, and we can use the reduced system \eqref{RedSIRS} for computing the steady state vector. From the equilibrium equation $( \beta S^* d_G - \delta)  I^*=0$, when $I^* \neq 0$, we have $S^*= \frac{\delta}{\beta d_G}$. Since $I^* < 1- S^*$, we obtain}

\sO{$$\frac{1}{\beta}\cdot \frac{\delta I^*}{S^*} \cdot \frac{1}{1-S^*} <  \frac{\delta}{\beta S^*}=d_G=\lambda_1(A).$$}
\ste{that contradicts \eqref{cond_maxD}.}\\\\

\sO{\begin{theorem}\label{attractivity}
Let $G=(V,E)$ be an undirected regular graph with degree $d_G$, and $\tau > \frac{\gamma + \sigma}{\gamma}\frac{1}{\lambda_1(A)}$. \\ Then, if $ \delta > \sigma$, the endemic equilibrium $ Y^*=(S^*_1, \ldots, S^*_N, I^*_1, \ldots, I_N^*, R_1^*, \ldots, R_N^*)$ is asymptotically stable in $\mathring{ \tilde{\Gamma}}$ and $\tilde{\Omega}\cap \mathring {\tilde{\Gamma}}$
is a subset of the domain of attraction of $Y^*$.
\end{theorem}}

\begin{proof}

\sO{The asymptotic stability is provided in Thm. \ref{eq0} b).
Thus, we have to prove that, under the condition $\delta > \sigma$, we can identify a subset of the domain of attraction of the endemic equilibrium in the case of regular graphs.}

We can apply Thm. \ref{thmRed} in the case of $G$ regular graph. 
From \eqref{RedSIRS}, and since $S+I+R=1$, we obtain


\begin{align}\label{IR}
\frac{d I}{dt} &= \beta \; d_G (1-I-R)  I - \delta  I, \\ \nonumber
\frac{d R}{dt} &= \delta  I- \gamma  R+ \sigma  (1-I-R), \nonumber
\end{align}
with initial conditions $(I(0),R(0)) \in \mathring \Gamma' $, with $\Gamma' = \{  (I,R) \in \mathbb{R}^2_+ | I+R \leq 1 \}.$ \sO{By Thm. \ref{thmRed}, $S_i(t)=S(t)$, $I_i(t)=I(t)$, $R_i(t)=R(t)$, $i=1, \ldots, N$, for all $t \geq 0$, since all nodes have the same trajectories when starting with the same initial conditions.} 

Now, let us consider the Volterra-type function, $U=I-I^*-I^* \ln (I/I^*)$ , 
used by many authors \cite{goh1978global, freedman1985global,beretta1986general,lin1993global}, and the common quadratic function $Z=\frac{1}{2}(R-R^*)^2$.
Since from the equilibrium equations $ \beta (1-I^*-R^*) d I^* - \delta  I^*=0$ and $\delta  I^*- \gamma  R^*+ \sigma  (1-I^*-R^*)=0$, after some manipulations, we obtain

$$U'=-\beta \;d_G(I-I^*)^2-\beta d_G(I-I^*)(R-R^*),$$

$$Z'=(\delta-\sigma)(R-R^*)(I-I^*)-(\gamma+\sigma)(R-R^*)^2.$$

Let us define $V=c U + Z$, where $c=\frac{ (\delta-\sigma)}{\beta d_G} > 0$. Then,
$$
V'=- (\delta-\sigma)(I-I^*)^2-(\gamma+\sigma)(R-R^*)^2,
$$
 so that $V'\le 0$ and $V'=0$ if and only if $I=I^*$ and $R=R^*$. \sO{Thus, $V$ is a Lyapunov function for the system \eqref{IR} and by a classical theorem of Lyapunov we have the global attractivity (and the local stability) of the endemic equilibrium $(I^*,R^*)$ in $\mathring \Gamma'$}.

\sO{Consequently, for a dynamics over a regular graph, when $Y(0) \in  \tilde{\Omega}\cap \mathring {\tilde{\Gamma}}$, the trajectories of the original system \eqref{3dim} coincide with the solution obtained from the reduced system \eqref{IR} (recalling that $S=1-I-R$).} 
 Hence, $ \lim_{t \to \infty}Y(t) = Y^* $,  where  $Y^* \in  \tilde{\Omega}\cap \mathring {\tilde{\Gamma}}$.
\end{proof}

\subsection{Notes on the basic SIRS epidemic model}\label{SIRS}

From Thm. \ref{eq0}, we can see that for a basic SIRS model, i.e., by setting $\sigma=0$, it holds

\begin{equation}\label{tauSIRS}
 \tau^{(1)}_{c;SIRS}=\frac{1}{\lambda_1(A)}.
\end{equation}

 \sO{Moreover, from Thm. \ref{eq0} b), above $ \tau^{(1)}_{c;SIRS}$ the asymptotic stability of the endemic equilibrium is always ensured, without further conditions.}
Let us note that

\begin{equation*}\label{eqtau}
\tau^{(1)}_{c;SIRS} =\tau^{(1)}_{c;SIS} =\tau^{(1)}_{c;SIR},
\end{equation*}
see, indeed, for the SIS and SIR threshold, e.g., \cite{VanMieghem2009,QEP,SIRscoglio}. 
Comparing \eqref{tauSIRS} with \eqref{thrSIRSV}, it is clear how the introduction of vaccination extends the region of extinction, that is values of $\delta$ and $\beta$ for which the epidemics would persist without vaccination can be instead sufficient to drop the epidemics if the vaccination is introduced in the population.
 \sO{The mean-field threshold for the SIRS model is not able to capture the role of $\gamma$ in the extinction and persistence of epidemics. 
 However, the value of $\gamma$ in the mean-field model explicitely influences the average fraction of infected nodes in the steady state, indeed
for the SIRS model the positive equilibrium point has the following components:}

\begin{equation*}\label{SIRSinf}
I_i^*= \frac{1}{1+\delta/\gamma}\left(1-\frac{1}{1+\tau (1+\delta/\gamma) \sum_{j=1}^N a_{ij}I^*_j}\right),
\end{equation*}

\begin{equation*}\label{SIRSrem_susc}
 R^*_i=\frac{\delta}{\gamma} I_i^*,  \qquad  S_i^*= 1- \left(\frac{\gamma + \delta}{\gamma}\right)I_i^*,
\end{equation*}
for $i=1, \ldots, N$. We can see that 
 for fixed values of $\beta$ and $\delta$, as $\gamma$ increases the steady state solution $I_i^*$ approaches that of the SIS model (see \cite{VanMieghem2009}). This is easy to understand since, as $\gamma$ increases, the average immune period tends to decrease (a removed individual quickly return to the susceptible state) and the behavior of the SIRS model approaches that of the SIS model.  Conversely, if the value of $\gamma$ goes down (the return to the
susceptible state is protracted) the probability of being infectious tends to decreases, detaching from the SIS steady-state solution \cite{Turri}.

\section{Numerical investigations}\label{numInv}

In Fig.~\ref{fig:SIRV_1}, we consider the average fraction of infected nodes 
 of the SIRS$_v$ model, as function of time and $\sigma$, for a complete graph with $N=50$, by fixing $\beta=0.25$, $\delta=0.4$ and $ \gamma=0.2$. We can see how increasing the value of the rate of vaccination $\sigma$, the average fraction of infected nodes decreases, thus passing from a region of persistence 
to a region of extinction. \sO{Thus, once known the topology of the contact network and the other parameters involved, we can calibrate the value of $\sigma$ to guide the epidemic towards the extinction}.


In Fig.~\ref{fig:Steady_1}, we report the steady-state average fraction of infected nodes, $\bar{I}^*$, as function of $\gamma$, for different values of $\sigma$, by considering a complete graph with $N=50$, $\beta=0.25$ and $\delta=0.9$. We can see that, by fixing the value of $\sigma$, the value of $\bar{I}^*$ increases as $\gamma$ increases, thus a shorter immunity period leads to a more aggressive epidemic. Vice versa, by fixing $\gamma$, the value of the prevalence in the steady-state clearly decreases as $\sigma$ increases. \sO{Thus, the less time each individual remains unvaccinated, 
the more the entire population will benefit in terms of percentage of infected individuals in the long-run}.

 \begin{figure}[h]
 \centering
 \includegraphics[width=9cm]{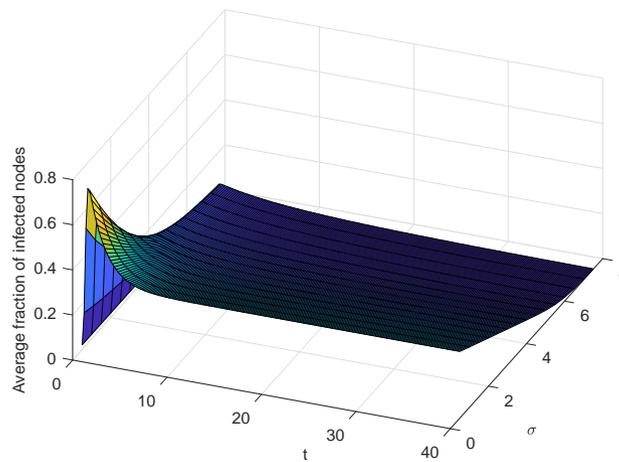} 
 \caption{ SIRS$_{v}$ average fraction of infected nodes as function of time and $\sigma$, for a complete graph with $N=50$, $\beta=0.25$, $\delta=0.4$, $ \gamma=0.2$. \sO{ At time 0 there is one infected node.}}
\label{fig:SIRV_1}
   \end{figure}

 \begin{figure}[h]
 \centering
 \includegraphics[width=9cm]{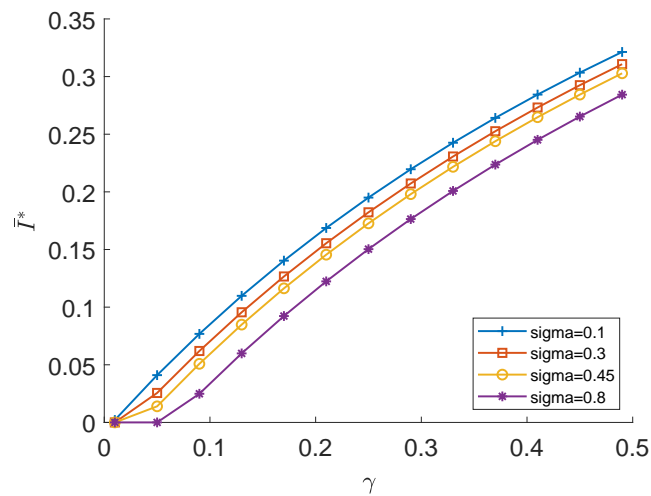}
 \caption{ SIRS$_{v}$ steady-state average fraction of infected nodes,  $\bar{I}^*$, as function of $\gamma$, for different values of $\sigma$, for a complete graph with $N=50$, $\beta=0.25$, $\delta=0.9$.}
\label{fig:Steady_1}
   \end{figure}

\sO{Fig. ~\ref{fig:EqPart} depicts the trajectories of the infection and recovery probabilities
 from system \eqref{3dim} for two nodes of a regular graph with $N=50$ and $d_G=10$, starting with different initial conditions, $I_i(0) \neq I_j(0)$, $R_i(0) \neq R_j(0)$,  for $i=1, \ldots N$ .
\ste{These solutions are compared with the one computed using the reduced system \eqref{RedSIRS},
considering all nodes having the same initial conditions equal to $ \ov{I}(0)=1/N \cdot \sum_{i=1}^N I_i$, and $\ov  R(0)=1/N \cdot \sum_{i=1}^N R_i$. }
We can see that trajectories starting outside the invariant set $\tilde \Omega$
  tend to approach the one starting in $\tilde \Omega$ as time goes on. It can be seen that, as pointed out in the Remark \ref{remSteady}, the positive equilibrium belongs to $\tilde{\Omega}\cap \mathring {\tilde{\Gamma}}$, and can be computed with the reduced system. Thus, from the numerical investigation, we can note that even when the initial conditions of the nodes are different, the trajectories are attracted by the endemic equilibrium}.

In Fig.~\ref{fig:Sim-SIRS_V1}, we consider a complete graph with $N=50$ and provide a comparison
between the dynamics of the prevalence, obtained from the solution of the ODE system \eqref{2dim}, and the averaging $2 \cdot 10^4$ sample paths resulting from the discrete event simulation of
the exact stochastic SIRS$_v$ process. 
 In Fig.~\ref{fig:Sim-SIRS_V1} a), we consider values of the parameters such that  $\tau <\tau^{(1)}_{c;SIRS_v}$ while in b) and c) values for which $\tau > \tau^{(1)}_{c;SIRS_v}$. We can see that in a) only in the early phase the approximated model is slightly above the exact averaged dynamics.
 \sO{In  b) for the chosen parameters values, i.e. $\beta=1$, $\delta=0.45$, $\gamma=0.2$, $\sigma=0.4$, there is a quite perfect match. Interestingly, in c) when we consider the same values for $\beta$, $\delta$ and $\sigma$, but $\gamma=0.06$ we have a different qualitative behavior between the exact and the approximated model after a certain point in time. Indeed, we can see that the exact prevalence, after reaching the peak, starts to decrease towards the state with no infected quite early, while in the approximate model, the infection remains persistent.}

In Fig.~\ref{fig:Sim-SIRS_V2} we report the same type of comparison done in Fig.~\ref{fig:Sim-SIRS_V1}, but for a regular graph with $N=50$ and $d_G=10$. We can see that, for the chosen parameters, the solution of the approximated model stays slightly
 above that of the exact model, thus providing an upper bound for the exact averaged dynamics. 
 However we can note that in a), as well as in b), for the time window considered, the qualitative behavior is the same between the two models. For the chosen values of the parameters in b), the stochastic dynamics seems to stand on a positive value for long time before reaching the absorption, resembling the behavior of the mean-field model that above the threshold reaches the positive equilibrium. 
However, in c), as for the complete graph case, when we have the same $\beta$, $\delta$ and $\sigma$, but a lower $\gamma$ than in b), the qualitative behavior between the dynamics of the two models is different and after the peak the exact dynamics reaches the extinction 
 quite early.
Thus, let us rewrite the condition for the extinction $\tau \leq \tau_{c;SIRS_{v}}^{(1)}$ in the following way

\begin{equation*}
\rho \leq \frac{1}{\lambda_1(A)}=\tau_c, \qquad \text{where} \qquad \rho=\frac{\beta \gamma}{\delta(\gamma + \sigma)}.
\end{equation*}

\sO{Then, we can assert that from Fig.~\ref{fig:Sim-SIRS_V1} and Fig.~\ref{fig:Sim-SIRS_V2}, in some $\rho$-region around $\tau_c$, we can observe deviations between the mean-field and the exact model. Thus we could expect that, in general, deviations between the two models are expected for intermediate value of $ \beta \gamma/(\delta(\gamma + \sigma))$. This behavior can also be observed in the SIS model in a $\tau$-region around $\tau_c$ \cite{VanMieghem2009}}.\\\\





 \begin{figure}[h]
 \centering
 \includegraphics[width=8cm]{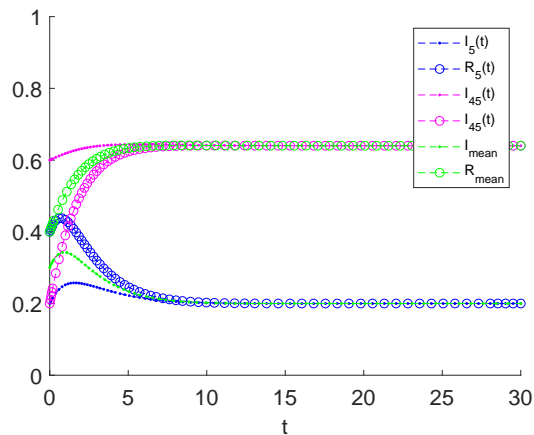}
 \caption{ Dynamics of infection and recovery probabilities of two selected nodes starting with different initial conditions, from system \eqref{3dim}, compared with those obtained from the reduced system \eqref{RedSIRS}, where all nodes have the same initial conditions. 
We consider a regular graph with $N=50$ and $d_G=10$ with $\beta=0.25$, $\delta=0.4$, $ \gamma=0.2$, $\sigma=0.3$,  $\tau > \tau_{c;SIRS_v}^{(1)}$. 
}
\label{fig:EqPart}
 \end{figure}


 \begin{figure}[h!]
 \centering
    \begin{minipage}{0.4\textwidth}
    \includegraphics[width=\textwidth,height=5cm]{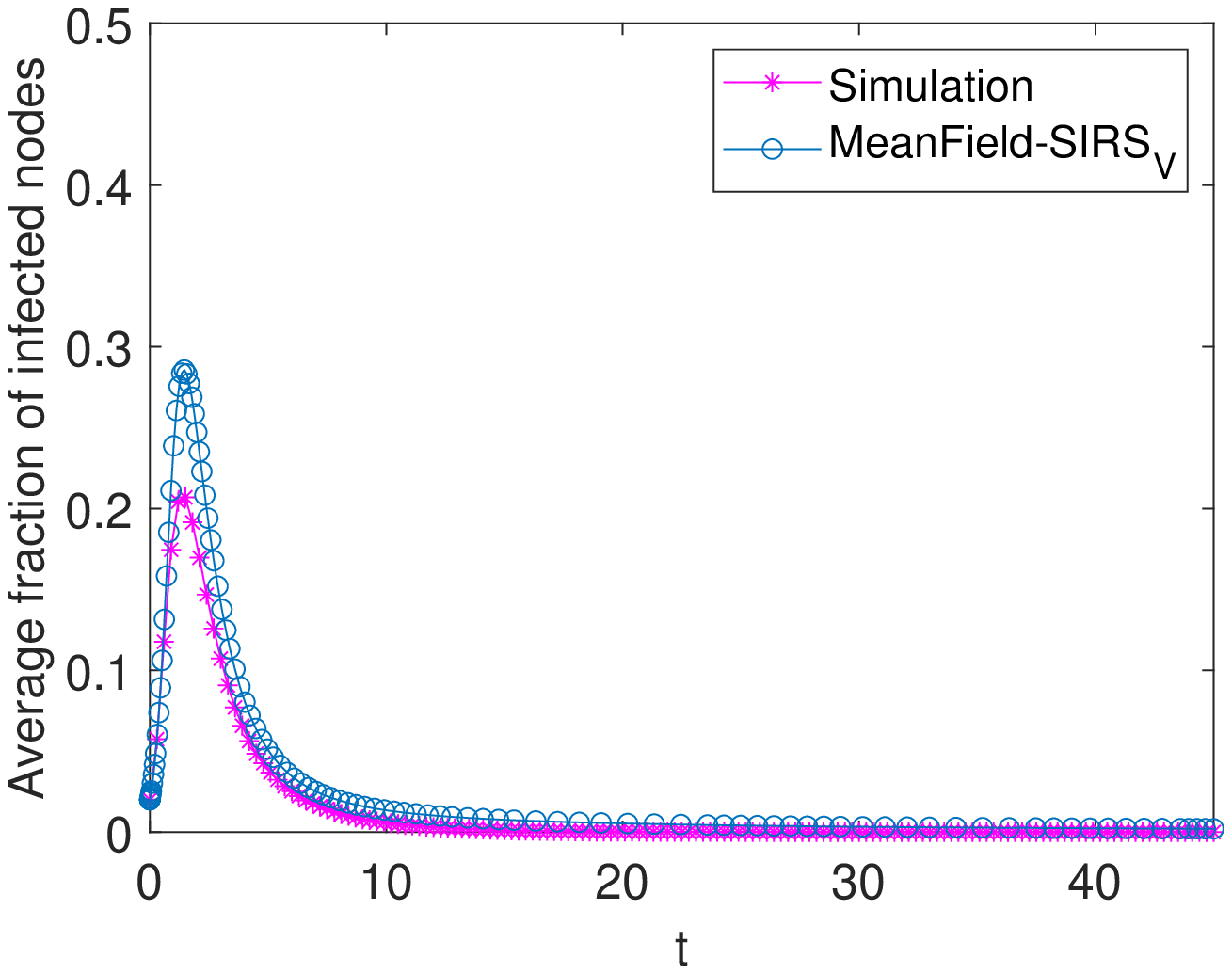}\put(-210,130){a)}
    \end{minipage}
   \hskip4mm
   \begin{minipage}{0.4\textwidth}
    \includegraphics[width=\textwidth,height=5cm]{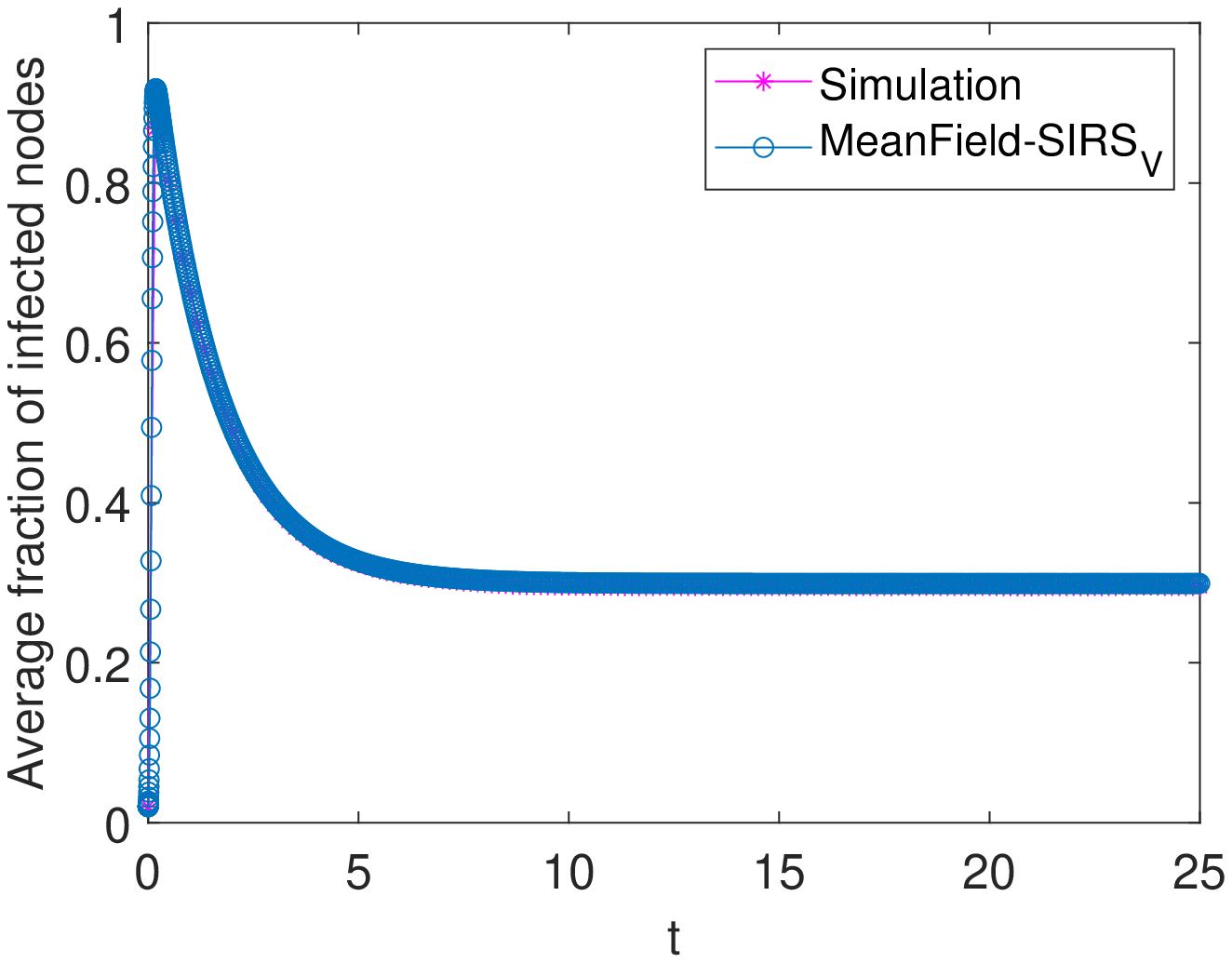}\put(-210,130){b)}
\end{minipage}
 \hskip4mm
   \begin{minipage}{0.4\textwidth}
    \includegraphics[width=\textwidth,height=5cm]{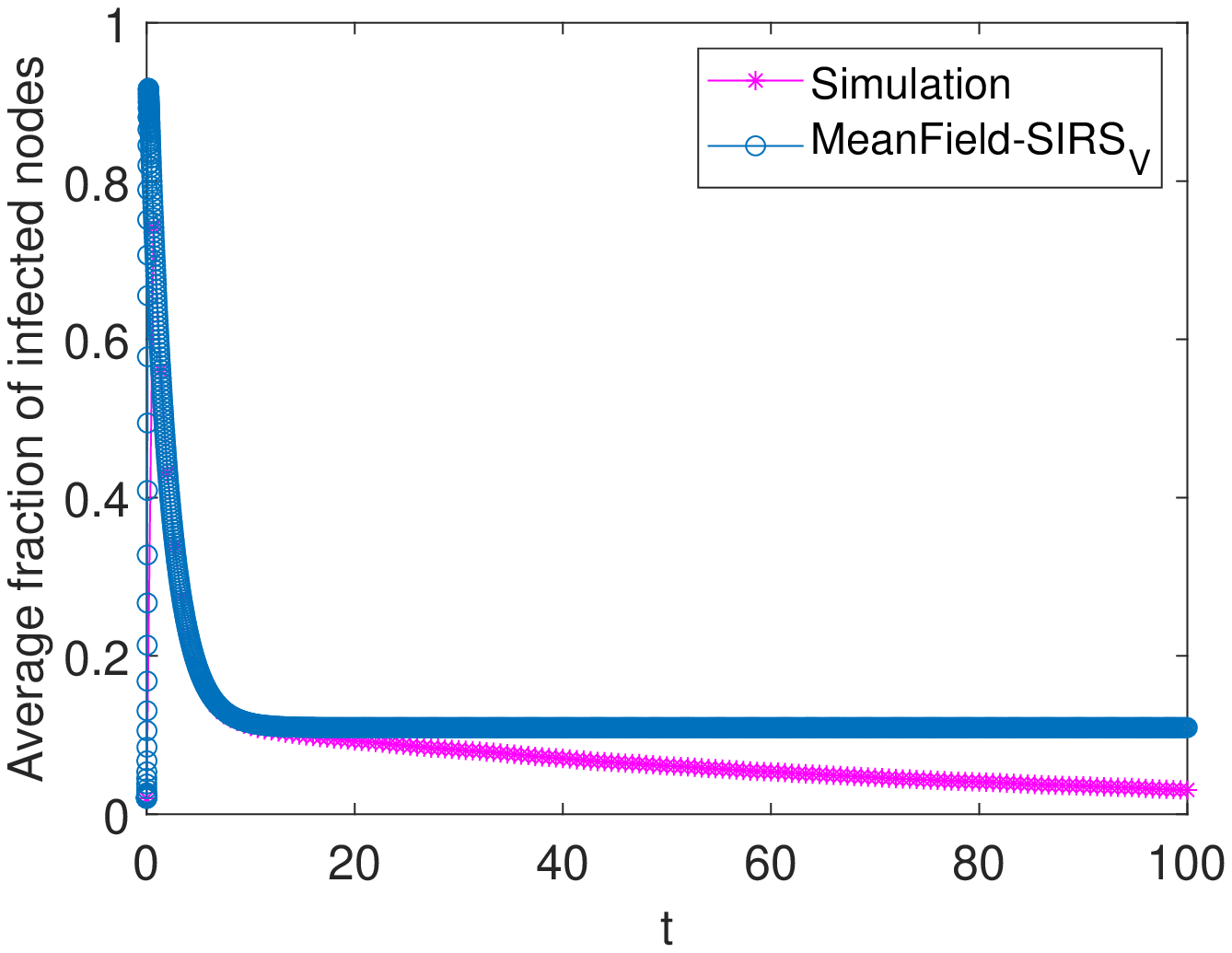}\put(-210,130){c)}
\end{minipage}
    \caption{ \ste{Comparison
between the dynamics of the prevalence for the SIRS$_v$ model, obtained from the numerical solution of \eqref{2dim}, and by averaging $2 \cdot 10^4$ simulated sample paths of
the stochastic process}. 
Complete graph with $N=50$, $\sigma=0.4$.  a) $\beta=0.1$, $\delta=0.9$, $ \gamma=0.1$, $\tau <\tau^{(1)}_{c;SIRS_v}$. b) ,  $\beta=1, \delta=0.45$, $ \gamma=0.2$, $\tau> \tau^{(1)}_{c;SIRS_v}$. c) $\beta=1, \delta=0.45$, $ \gamma=0.06$, $\tau> \tau^{(1)}_{c;SIRS_v}$.  \sO{At time 0 there is one infected node.}}
    \label{fig:Sim-SIRS_V1}
  \end{figure}

%
%
%

 \begin{figure}[h!]
 \centering
    \begin{minipage}{0.4\textwidth}
    \includegraphics[width=\textwidth,height=5cm]{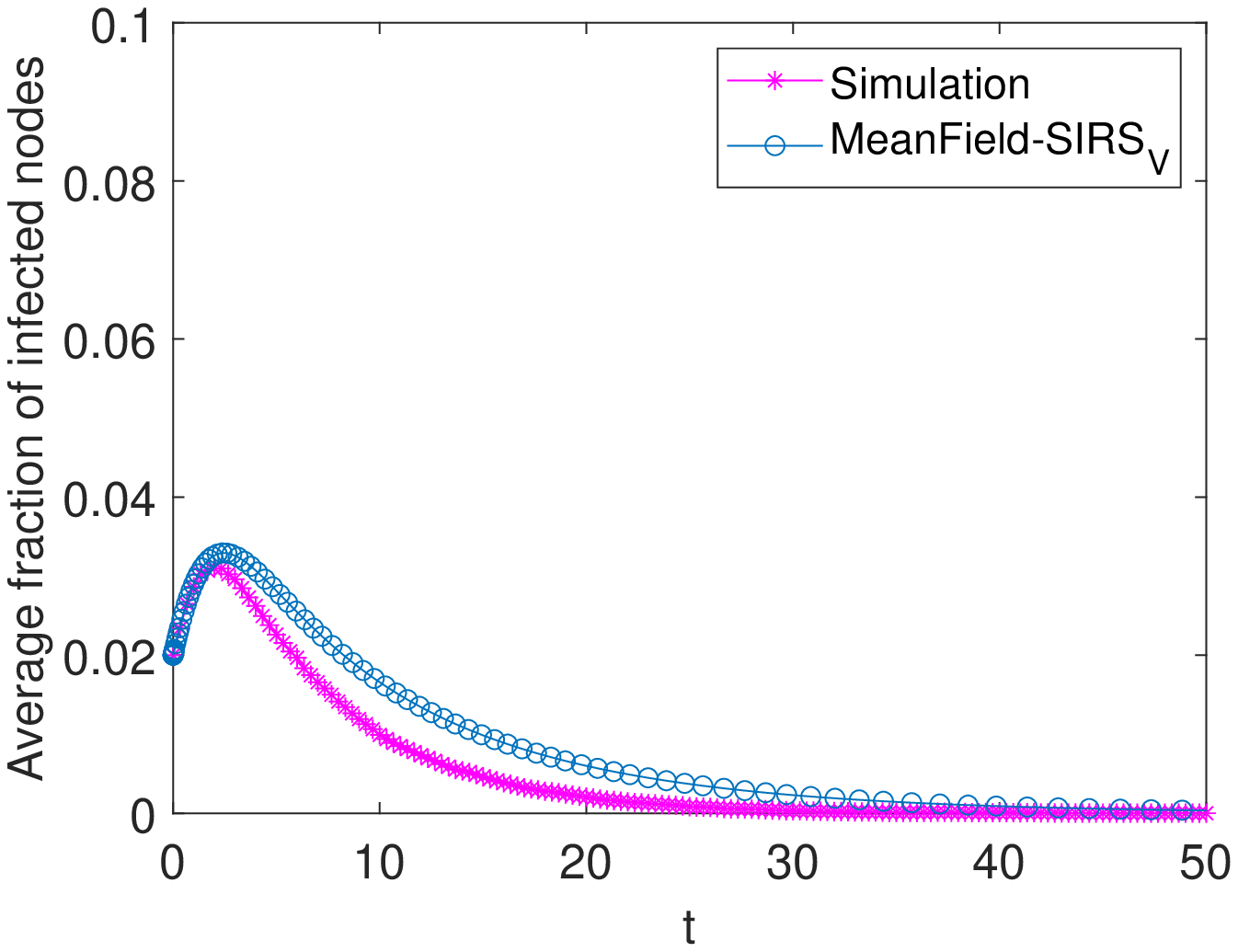}\put(-210,130){a)}
    \end{minipage}
   \hskip4mm
   \begin{minipage}{0.4\textwidth}
    \includegraphics[width=\textwidth,height=5cm]{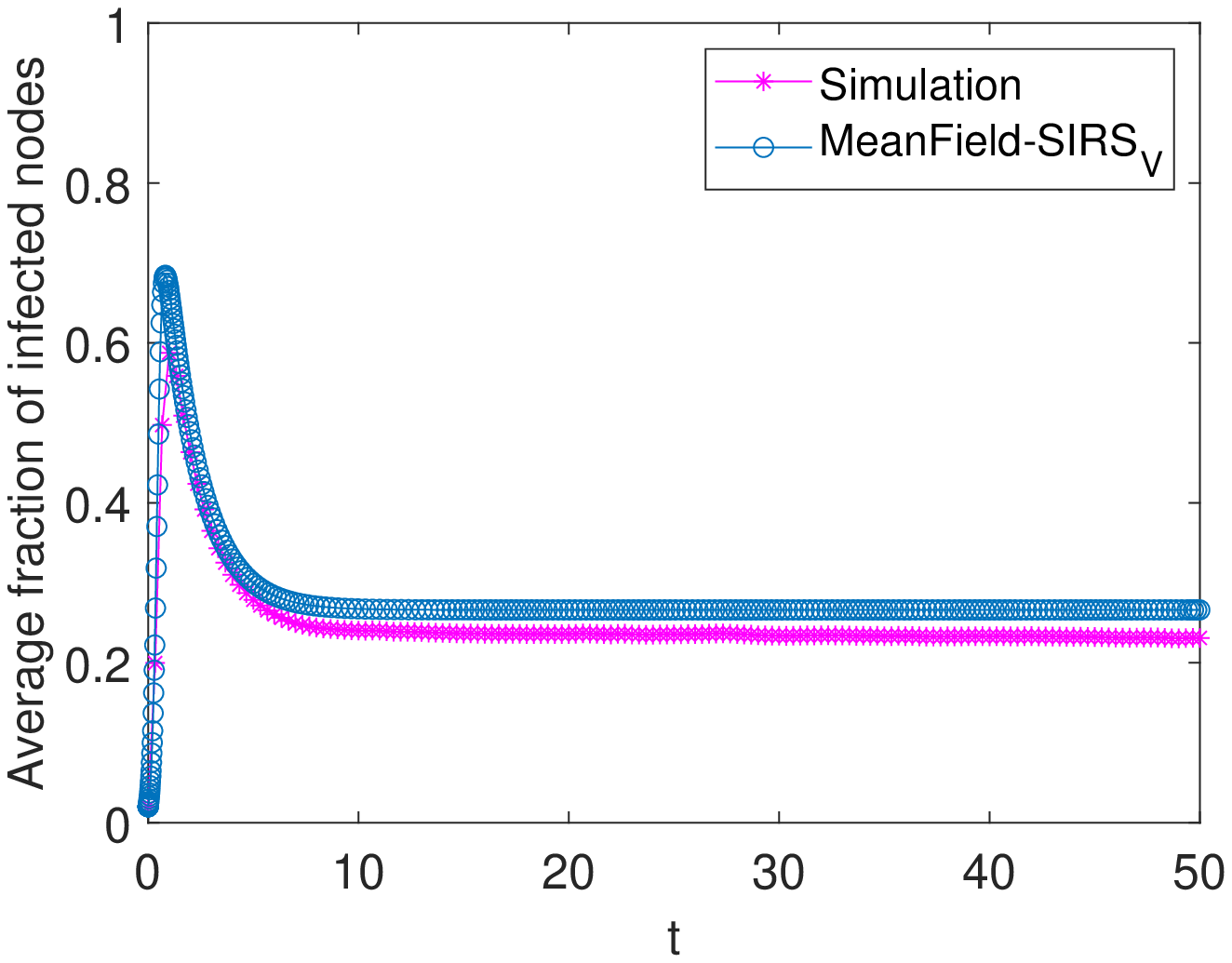}\put(-210,130){b)}
\end{minipage}
 \hskip4mm
   \begin{minipage}{0.4\textwidth}
    \includegraphics[width=\textwidth,height=5cm]{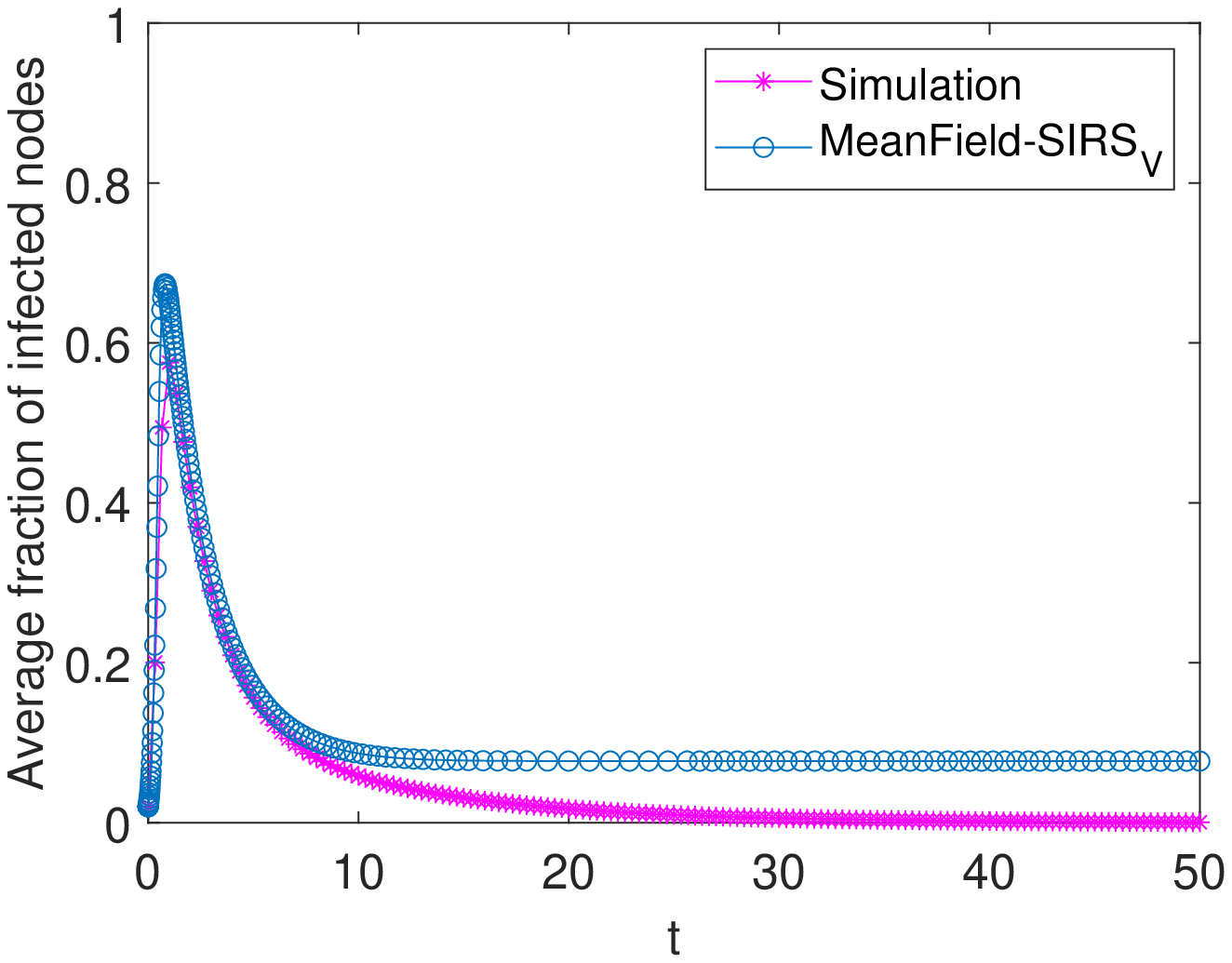}\put(-210,130){c)}
\end{minipage}
    \caption{ \ste{Comparison
between the dynamics of the prevalence for the SIRS$_v$ model, obtained from the numerical solution of \eqref{2dim}, and by averaging $2 \cdot 10^4$ simulated sample paths of
the stochastic process}. Regular graph with $N=50$ and $d_G=10$,  $\sigma=0.45$. a) $\beta=0.1$, $\delta=0.4$, $ \gamma=0.2$, $\tau<\tau^{(1)}_{(c;SIRS_v)}$  b) $\beta=1$, $\delta=0.4$, $ \gamma=0.2$, $\tau>\tau^{(1)}_{(c;SIRS_v)}$.  c) $\beta=1, \delta=0.4$, $ \gamma=0.06$, $\tau> \tau^{(1)}_{(c;SIRS_v)}$. At time 0 there is one infected node.}
    \label{fig:Sim-SIRS_V2}
  \end{figure}



%


\section{Conclusion}

In this work, we started by considering the exact stochastic Markov description of a SIRS model, with vaccination, on networks. In this context, we investigated the mean time of the epidemic. \ste{We found a 
sufficient condition, related to the topological properties of the network and to the model parameters, for the fast extinction  (no more infected), for avoiding a long-term persistence. We provided also some numerical investigations to assess the role of the immunity-loss parameter in the extinction mean time.}


Starting from a node-level description of the exact Markov process, that becomes neither analytically
nor computationally tractable with increasing number of nodes $N$, we derived an approximation of it by means of a first-order meanfield technique.
 We obtained a set of $3N$ nonlinear
differential equations, specifying the state probabilities of each node.
\ste{At this point, we focused on the stability properties of the approximated model. We start considering the stability analysis provided in
\cite{yang2017heterogeneous}, where the authors consider an heterogeneous version of our SIRS$_{v}$ model.
Based on their analysis, we provided the critical threshold, which separates an extinction region from an endemic one, in terms of the parameters and the network topology. In this way, it is made explicit to what extend the threshold and the steady-state solutions are influenced by the value of the immunity-loss parameter, and by the introduction of the vaccination, comparing the results with the basic SIRS model.}

\ste{A noteworthy aspect is that, to the best of our knowledge, the global asymptotic stability (GAS) of the endemic equilibrium for a SIRS model of our kind is still an open question, in that there are only partial results with additional strict conditions on the model parameters. Also in \cite{yang2017heterogeneous} additional  restrictions, dependent on the network topology, are imposed on the model parameters in order to obtain a sufficient condition for the GAS of the endemic equilibrium.
However, we have not been able to find graphs and parameters for which this condition is valid, and even in \cite{yang2017heterogeneous} the authors do not provide numerical examples in which the condition holds. Moreover, 
we show, that in the homogeneous setting, it is never satisfied in the case of regular graphs. }
\ste{For this reason, we analyzed the domain of attraction of the positive equilibrium, at least for these specific graphs, using  the notion of equitable partitions.} 

First, we proved the existence of a positively invariant set for the system when a graph posses an equitable partition, for which, when the initial conditions belong to this set, the whole epidemic dynamics can be expressed by a reduced system with respect to the starting one \eqref{2dim}. This reduced system can be used for the computation of the endemic equilibrium that belongs to this invariant set (see Remark \ref{remSteady}). \ste{This result is interesting in its own right, since it extends to a SIRS-type model what have been found for the SIS model in \cite{QEP}}.

Since a regular graph is a special case of graph with equitable partition, we showed that, when the recovery rate is higher than the vaccination rate, the aforementioned invariant set is contained in the domain of attraction of the endemic equilibrium (Thm. \ref{attractivity}). 

\ste{We also provided numerical investigations. First, we investigated the influence of the {immunity-loss} rate $\gamma$ and the vaccination rate $\sigma$ on the dynamics of the prevalence, and hence on the steady-state average fraction of infected nodes. Once fixed the graph and the other model parameters involved, we increased the value of $\sigma$, and as expected, the steady-state average fraction of infected nodes decreases. 
to a region of extinction. Vice versa, fixed $\sigma$, the increase of $\gamma$ leads to an increased steady-state prevalence, highlighting how a shorter immunity period leads to a more aggressive epidemic.}

\ste{Then, we reported the dynamics of infection and recovery probabilities of two selected nodes in a regular graph, and compared the trajectories of these probabilities starting outside the invariant set containing the endemic equilibrium \eqref{inv_set} (not all nodes have equal initial conditions) with those starting within the invariant set (all nodes have the same initial conditions). This numerical simulation says us something more than Thm. \ref{attractivity}, namely, also trajectories that start outside the invariant set, as time goes on, tend to approach those starting within, and finally reach the endemic equilibrium.}

\ste{ Finally, we compared the average behaviour of the exact stochastic SIRS$_{v}$ model with the approximated one. We can see, how, as we conjectured, the dynamics of the prevalence of the mean-field model tends stays slightly
 above that of the stochastic one, but specially for the case of the complete graph, in some region parameters, there si quite a perfect match between the two dynamics. However, there are a region parameter where a
different qualitative behavior between the two model appears after a certain point in time. Indeed, the exact prevalence, after reaching the peak, starts to decrease towards the state with no infected quite early, while in the approximate model, the infection remains persistent.}

\ste{In view of our results, we think that there are a lot of interesting aspects about the SIRS model (with vaccination) on network that can be further investigated, both for the stochastic case, and its mean-field approximation.}

\section*{Acknowledgment}

The research of Stefania Ottaviano was supported by ISSTN, Istituto di Scienze della Sicurezza, University of Trento. \\
This work does not have any conflicts of interest.


\begin{thebibliography}{10}

\bibitem{kribs2000simple}
Kribs-Zaleta Christopher~M, Velasco-Hern{\'a}ndez Jorge~X. A simple vaccination
  model with multiple endemic states.  {\it Mathematical biosciences.
  }2000;164(2):183--201.

\bibitem{alex2004vaccination}
Alexander Murray~E, Bowman Christopher, Moghadas Seyed~M, Summers Randy, Gumel
  Abba~B, Sahai Beni~M. A vaccination model for transmission dynamics of
  influenza.  {\it SIAM Journal on Applied Dynamical Systems.
  }2004;3(4):503--524.

\bibitem{elbasha2006theoretical}
Elbasha Elamin~H, Gumel Abba~B. Theoretical assessment of public health impact
  of imperfect prophylactic HIV-1 vaccines with therapeutic benefits.  {\it
  Bulletin of mathematical biology. }2006;68(3):577.

\bibitem{sun2010global}
Sun Chengjun, Yang Wei. Global results for an SIRS model with vaccination and
  isolation.  {\it Nonlinear Analysis: Real World Applications.
  }2010;11(5):4223--4237.

\bibitem{cai2018global}
Cai Li-Ming, Li~Zhaoqing, Song Xinyu. Global analysis of an epidemic model with
  vaccination.  {\it Journal of Applied Mathematics and Computing.
  }2018;57(1-2):605--628.

\bibitem{boccaletti2006}
Boccaletti Stefano, Latora Vito, Moreno Yamir, Chavez Martin, Hwang D-U.
  Complex networks: Structure and dynamics.  {\it Physics reports.
  }2006;424(4):175--308.

\bibitem{PietSurvey}
Pastor-Satorras Romualdo, Castellano Claudio, Van~Mieghem Piet, Vespignani
  Alessandro. Epidemic processes in complex networks.  {\it arXiv preprint
  arXiv:1408.2701. }2014;.

\bibitem{kiss2017mathe}
Kiss Istv{\'a}n~Z, Miller Joel~C, Simon P{\'e}ter. Mathematics of epidemics on
  networks.  {\it Cham: Springer. }2017;598.

\bibitem{pastor2015epidemic}
Pastor-Satorras Romualdo, Castellano Claudio, Van~Mieghem Piet, Vespignani
  Alessandro. Epidemic processes in complex networks.  {\it Reviews of modern
  physics. }2015;87(3):925.

\bibitem{yang2017heterogeneous}
Yang Luxing, Draief Moez, Yang Xiaofan. Heterogeneous virus propagation in
  networks: a theoretical study.  {\it Mathematical Methods in the Applied
  Sciences. }2017;40(5):1396--1413.

\bibitem{balthrop2004technological}
Balthrop Justin, Forrest Stephanie, Newman Mark~EJ, Williamson Matthew~M.
  Technological networks and the spread of computer viruses.  {\it Science.
  }2004;304(5670):527--529.

\bibitem{danon2011networks}
Danon Leon, Ford Ashley~P, House Thomas, et al. Networks and the epidemiology
  of infectious disease.  {\it Interdisciplinary perspectives on infectious
  diseases. }2011;2011.

\bibitem{skwara2018superdiffusion}
Skwara Urszula, Mateus Lu{\'\i}s, Filipe Raquel, Rocha Filipe, Aguiar
  Ma{\'\i}ra, Stollenwerk Nico. Superdiffusion and epidemiological spreading.
  {\it Ecological Complexity. }2018;36:168--183.

\bibitem{Turri}
Bonaccoris Stefano, Turri Silvia. Deterministic and Stochastic Mean-Field SIRS
  Models on Heterogeneous Networks.  In: Discrete and Continuous Models in the
  Theory of Networks, Springer, 2020; 281:67--89.

\bibitem{chen2014SirsVaccNet}
Chen Lijuan, Sun Jitao. Global stability and optimal control of an SIRS
  epidemic model on heterogeneous networks.  {\it Physica A: Statistical
  Mechanics and its Applications. }2014;410:196--204.

\bibitem{liu2019SIRSnetl}
Liu Lijun, Wei Xiaodan, Zhang Naimin. Global stability of a network-based SIRS
  epidemic model with nonmonotone incidence rate.  {\it Physica A: Statistical
  Mechanics and its Applications. }2019;515:587--599.

\bibitem{liu2017analysis}
Liu Qiming, Sun Meici, Li~Tao. Analysis of an SIRS epidemic model with time
  delay on heterogeneous network.  {\it Advances in Difference Equations.
  }2017;2017(1):309.

\bibitem{yu2015dynamical}
Yu~Rongzhong, Li~Kezan, Chen Baidi, Shi Dingqin. Dynamical analysis of an SIRS
  network model with direct immunization and infective vector.  {\it Advances
  in Difference Equations. }2015;2015(1):116.

\bibitem{PVM_GEMF}
Sahneh Faryad Darabi, Scoglio Caterina, Van~Mieghem Piet. Generalized Epidemic Mean-Field Model
  for Spreading Processes over Multi-Layer Complex Networks.  {\it IEEE/ACM
  Tran. on Networking. }2013;21(5):1609-1620.

\bibitem{VanMieghem2009}
Van~Mieghem P., Omic J., Kooij R.. Virus Spread in Networks.  {\it Networking,
  IEEE/ACM Tran. on. }2009;17(1):1-14.

\bibitem{SIRscoglio}
Youssef Mina, Scoglio Caterina. An individual-based approach to SIR epidemics
  in contact networks.  {\it Journal of theoretical biology.
  }2011;283(1):136--144.

\bibitem{lin1993global}
Lin Xiaodong, So~Joseph W-H. Global stability of the endemic equilibrium and
  uniform persistence in epidemic models with subpopulations.  {\it The ANZIAM
  Journal. }1993;34(3):282--295.

\bibitem{muroya2013global}
Muroya Yoshiaki, Enatsu Yoichi, Kuniya Toshikazu. Global stability for a
  multi-group SIRS epidemic model with varying population sizes.  {\it
  Nonlinear Analysis: Real World Applications. }2013;14(3):1693--1704.

\bibitem{mena1992dynamic}
Mena-Lorcat Jaime, Hethcote Herbert~W. Dynamic models of infectious diseases as
  regulators of population sizes.  {\it Journal of Mathematical Biology.
  }1992;30(7):693--716.

\bibitem{shuai2013global}
Shuai Zhisheng, Driessche Pauline. Global stability of infectious disease
  models using Lyapunov functions.  {\it SIAM Journal on Applied Mathematics.
  }2013;73(4):1513--1532.

\bibitem{godsil}
Godsil Christopher D., McKay Brendan D.  Feasibility conditions for the existence of walk-regular graphs.
  {\it Linear Algebra and its Applications. }1980;30:15-61.

\bibitem{QEP}
Bonaccorsi Stefano, Ottaviano Stefania, Mugnolo Delio, {De Pellegrini}
  Francesco. Epidemic Outbreaks in Networks with Equitable or Almost-Equitable
  Partitions.  {\it SIAM Journal of Applied Mathematics. }2015;75(6):2421 --
  2443.

\bibitem{bremaud1999markov}
Br{\'e}maud Pierre. {\it Markov chains: Gibbs fields, Monte Carlo simulation,
  and queues}.
\newblock Springer-Verlag New York; 1999.

\bibitem{van2014upper}
Van~Mieghem Piet, Sahnehz Faryad~Darabi, Scoglio Caterina. An upper bound for
  the epidemic threshold in exact Markovian SIR and SIS epidemics on networks.
 53rd IEEE Conference on Decision and Control. IEEE, 2014; 6228--6233.

\bibitem{draief2010epidemics}
Draief Moez, Massouli Laurent. {\it Epidemics and rumours in complex networks}.
\newblock Cambridge University Press; 2010.

\bibitem{donnelly1993correlation}
Donnelly Peter. The correlation structure of epidemic models.  {\it
  Mathematical biosciences. }1993;117(1-2):49--75.

\bibitem{cator2014nodal}
Cator Eric, Van~Mieghem Piet. Nodal infection in Markovian
  susceptible-infected-susceptible and susceptible-infected-removed epidemics
  on networks are non-negatively correlated.  {\it Physical Review E.
  }2014;89(5):052802.

\bibitem{cator2018reply}
Cator Eric, Donnelly Peter, Van~Mieghem Piet. Reply to Comment on Nodal
  infection in Markovian susceptible-infected-susceptible and
  susceptible-infected-removed epidemics on networks are non-negatively
  correlated.  {\it Physical Review E. }2018;98(2):026302.

\bibitem{Schwenk}
Schwenk Allen~J. Computing the characteristic polynomial of a graph.  In:
  Graphs and combinatorics, Springer 1974 (pp. 153--172).

\bibitem{EvolDelio}
Mugnolo Delio. {\it Semigroup methods for evolution equations on networks}.
\newblock Springer; 2014.

\bibitem{pellis2015seven}
Ball Frank, Britton Tom, House Thomas, et al. Seven challenges for
  metapopulation models of epidemics, including households models.  {\it
  Epidemics. }2015;10:63--67.

\bibitem{ottaviano2018optimal}
Ottaviano Stefania, De~Pellegrini Francesco, Bonaccorsi Stefano, Van~Mieghem
  Piet. Optimal curing policy for epidemic spreading over a community network
  with heterogeneous population.  {\it Journal of Complex Networks.
  }2018;6(5):800--829.

\bibitem{ottaviano2019community}
Ottaviano Stefania, De~Pellegrini Francesco, Bonaccorsi Stefano, Mugnolo Delio,
  Van~Mieghem Piet. Community Networks with Equitable Partitions.  In:
  Multilevel Strategic Interaction Game Models for Complex Networks, Springer
  2019 (pp. 111--129).

\bibitem{neuberger2020}
Neuberger John~M, Sieben Nandor, Swift James~W. Invariant synchrony subspaces
  of sets of matrices.  {\it SIAM Journal on Applied Dynamical Systems.
  }2020;19(2):964--993.

\bibitem{goh1978global}
Goh Bean-San. Global stability in a class of prey-predator models.  {\it Bulletin of
  Mathematical Biology. }1978;40(4):525--533.

\bibitem{freedman1985global}
Freedman Herbert I., So~Joseph W.-H. Global stability and persistence of simple food chains.
  {\it Mathematical biosciences. }1985;76(1):69--86.

\bibitem{beretta1986general}
Beretta Edoardo, Capasso Vincenzo. On the general structure of epidemic systems. Global
  asymptotic stability.  {\it Computers \& Mathematics with Applications.
  }1986;12(6):677--694.

\end{thebibliography}
\end{document}